\documentclass[letter,11pt]{article}

\usepackage[margin=1.2 in, top=1 in, bottom= 1.2 in]{geometry}

\usepackage[normalem]{ulem}

\usepackage{tikz}
\usetikzlibrary{hobby}
\usepackage{tikz-cd}
\usepackage{adjustbox}
\usepackage{fbox}

\usetikzlibrary{matrix, calc, arrows}

\usepackage{bibentry}
\usepackage[linktocpage=true]{hyperref}

\usepackage{lipsum}
\newcommand\blfootnote[1]{%
  \begingroup
  \renewcommand\thefootnote{}\footnote{#1}%
  \addtocounter{footnote}{-1}%
  \endgroup
}

\makeatletter
\g@addto@macro\normalsize{
  \setlength\abovedisplayskip{8pt}
  \setlength\belowdisplayskip{8pt}
  \setlength\abovedisplayshortskip{8pt}
  \setlength\belowdisplayshortskip{8pt}
  }
\makeatother

\usepackage{tocloft}

\usepackage[english]{babel}
\usepackage{amsfonts}
\usepackage{mathrsfs}
\usepackage{bbm}
\usepackage{latexsym}
\usepackage{math dots}
\usepackage{amssymb}
\usepackage{mathtools}

\usepackage{enumitem}
\setlist{nolistsep}
\usepackage{amsthm}
\usepackage[capitalize]{cleveref}
\crefformat{footnote}{#2\footnotemark[#1]#3} 

\newcommand\eqnitem[1][]{%
  \ifx\relax#1\relax  \item \else \item[#1] \fi
  \abovedisplayskip=0pt\abovedisplayshortskip=0pt~\vspace*{-\baselineskip}}

\newtheoremstyle{plain}{3mm}{3mm}{\slshape}{}{\bfseries}{.}{.5em}{}
\newtheoremstyle{definition}{2mm}{2mm}{}{}{\bfseries}{.}{.5em}{}
\theoremstyle{plain}
	
\newtheorem{theorem}{Theorem}

\newtheorem{lemma}[theorem]{Lemma}

\newtheorem{question}[theorem]{Question}
\theoremstyle{definition}
\newtheorem{definition}[theorem]{Definition}
\newtheorem{remark}[theorem]{Remark}
\newtheorem{example}[theorem]{Example}

\theoremstyle{plain}
\newcounter{MainTheoremCounter}

\newtheorem{Maintheorem}[MainTheoremCounter]{Theorem}

\theoremstyle{plain}
\newtheorem*{namedthm}{\namedthmname}
\newcounter{namedthm}
	

\usepackage{chngcntr}
\counterwithin{theorem}{section}

\numberwithin{equation}{section}

\allowdisplaybreaks

\usepackage{xcolor}
\definecolor{Color2}{rgb}{0.78, 0.11, 0.0}
\hypersetup{citecolor = black,colorlinks,
			linkcolor = black,
			urlcolor = Color2}

\usepackage{titlesec}
\titleformat{\section}
  {\Large\center\bfseries}
  {\thesection.}{.7em}{}
\titlespacing*{\section}{0pt}{3.5ex plus 0ex minus 0ex}{1.5ex plus 0ex}
\titleformat{\subsection}
  {\large\center\bfseries}
  {\thesubsection.}{.7em}{}
\titlespacing*{\subsection}{0pt}{3.5ex plus 0ex minus 0ex}{1.5ex plus 0ex}
\titleformat{\subsubsection}
  {\center\bfseries}
  {\thesubsubsection.}{.7em}{}
\titlespacing*{\subsubsection}{0pt}{3.5ex plus 0ex minus 0ex}{1.5ex plus 0ex}

\addto\captionsenglish{}

\usepackage{titling}
\usepackage{authblk}

\newcommand{\Cech}{\v{C}ech}

\newcommand{\Szemeredi}{Szemer\'{e}di}

\newcommand{\supp}{{\normalfont\text{supp}}\,}

\newcommand{\eps}{\epsilon}
\newcommand{\N}{\mathbb{N}}
\newcommand{\Z}{\mathbb{Z}}
\newcommand{\R}{\mathbb{R}}

\newcommand{\Q}{\mathbb{Q}}
\newcommand{\T}{\mathbb{T}}

\newcommand{\defeq}{\vcentcolon=}

\newcommand\restr[2]{{ \left.\kern-\nulldelimiterspace #1 \right|_{#2}}}

\renewcommand{\epsilon}{\varepsilon}
\renewcommand{\leq}{\leqslant}
\renewcommand{\geq}{\geqslant}
\renewcommand{\setminus}{\backslash}

\usepackage[normalem]{ulem}

\DeclareMathOperator{\upclose}{\uparrow}
\DeclareMathOperator{\bigupclose}{\big\uparrow}

\newcommand{\syndetic}{\mathcal{S}}
\newcommand{\thick}{\mathcal{T}}

\newcommand{\dsyndetic}{d\mathcal{S}}

\newcommand{\dcsyndetic}{dc\mathcal{S}}
\newcommand{\dcthick}{dc\mathcal{T}}

\newcommand{\PS}{\mathcal{PS}}

\newcommand{\cubes}{\mathcal{Q}}

\newcommand{\central}{\mathcal{C}}
\newcommand{\IP}{\mathcal{IP}}

\newcommand{\filter}{\mathscr{F}}
\newcommand{\filtertwo}{\mathscr{G}}

\newcommand{\class}{\mathscr{F}}

\newcommand{\family}{\mathscr{F}}
\newcommand{\familyone}{\mathscr{F}}
\newcommand{\familytwo}{\mathscr{G}}
\newcommand{\familythree}{\mathscr{H}}

\newcommand{\odometer}{Z}
\newcommand{\orotation}{\rho}

\newcommand{\polyret}{R}

\newcommand{\FS}{\text{FS}}

\newcommand{\bpunch}{\beta}
\newcommand{\apunch}{\alpha}

\begin{document}

\title{Dynamically syndetic sets and the combinatorics of syndetic, idempotent filters}

\author[1]{Daniel Glasscock}
\author[2]{Anh N. Le}

\affil[1]{\small Dept. of Mathematics and Statistics, U. of Massachusetts Lowell, Lowell, MA, USA}
\affil[2]{\small Dept. of Mathematics,
	U. of Denver, Denver, CO, USA}

\date{}

\maketitle




\begin{abstract}
A subset of the positive integers is \emph{dynamically central syndetic} if it contains the times that a point returns to a neighborhood of itself in a minimal topological dynamical system.  These sets are part of the highly-influential link between dynamics and combinatorics forged by Furstenberg and Weiss in the 1970's.  Our main result is a characterization of dynamically central syndetic sets as precisely those sets that belong to syndetic, idempotent filters.  This gives a ``global'' analogue to the well-known ``local'' characterization of Furstenberg's central sets as members of minimal, idempotent ultrafilters.  Applying the main result, we answer two open questions posed by Host, Kra, and Maass concerning sets of pointwise topological recurrence.
\end{abstract}

\setcounter{tocdepth}{2}
\tableofcontents

\section{Introduction}
\label{sec:intro}

A subset of the positive integers, $\N$, is called \emph{dynamically syndetic} if it contains a set of the form
\[
    R(x,U) \defeq \big\{ n \in \N \ \big| \ T^n x \in U \big\}, \quad x \in X, \quad U \subseteq X \text{ nonempty, open},
\]
where $(X,T)$ is a minimal topological dynamical system; see \cref{sec_top_dynamics}.
If $x \in U$, such a set is called \emph{dynamically central syndetic}.  Dynamically syndetic and dynamically central syndetic sets are the main objects of study in this paper.

Examples of dynamically syndetic sets abound in combinatorics and number theory.  Periodic sets and their almost-periodic and higher-order generalizations, namely Bohr and nil-Bohr sets, are dynamically syndetic.  The images of Beatty sequences and the set of those positive integers with even $2$-adic valuation are examples of dynamically central syndetic sets. We give further examples in \cref{sec_dsyndetic} and an instructive non-example at the beginning of \cref{sec_intro_char_of_ds_sets} below.

\subsection{Motivation: structure in dynamics and combinatorics}
\label{sec_part_one_intro}

Dynamically syndetic sets are structural objects that form part of the bridge between dynamics and combinatorics.  To motivate their study, we will describe two highly influential sets of results that exemplify this bridge.

Furstenberg and Weiss \cite{Furstenberg-ErgodicBehavior, Furstenberg_Weiss-topological-dynamics} pioneered the use of ergodic theory and topological dynamics as tools in Ramsey theory and combinatorial number theory.  A \emph{correspondence principle} -- a device used to create a dynamical object out of a combinatorial one -- makes available a wide range of tools and ideas from dynamics.
For example, a topological correspondence principle can be used to reduce van der Waerden's theorem -- ostensibly concerning arbitrary finite partitions of $\N$ in its original formulation \cite{van_der_waerden_1927} or arbitrary syndetic subsets of $\N$ in an equivalent formulation \cite{kakeya_morimoto_1930} -- to a statement concerning dynamically syndetic sets.
The first step in Furstenberg's highly influential proof of \Szemeredi{}'s theorem \cite{Furstenberg-ErgodicBehavior} is also of this nature, transferring the problem from combinatorics to one about measure preserving systems.  This set of ideas has had a huge impact on the field and continues to yield dividends \cite{ward_ergodic_theory_interactions_survey}.

Rotations on compact abelian groups are among the simplest examples of dynamical systems, and the dynamically syndetic sets they generate, \emph{Bohr sets}, are  among the simplest examples of dynamically syndetic sets.  Nilsystems -- algebraic generalizations of compact group rotations -- give rise to dynamically syndetic sets called \emph{nil-Bohr sets}.
Since groundbreaking work of Host and Kra \cite{host_kra_2005} and Ziegler \cite{ziegler_2007}, nilsystems and nil-Bohr sets have been used to describe the ``structured'' components in powerful decomposition theorems in ergodic theory \cite{Bergelson_Host_Kra05, frant15, leibman_2010}, additive combinatorics and number theory \cite{Green_Tao10, Green_Tao12, Green_Tao12a, Green_Tao_Ziegler12} and, recently, in topological dynamics \cite{glasner_huang_shao_weiss_ye_2020}.
Thus, dynamically syndetic sets arising from highly structured systems have become essential in defining structure, even in non-dynamical settings.

In their full generality, the families of dynamically (central) syndetic sets have appeared explicitly in the literature in a few places.  In various combinations, Dong, Huang, Shao, and Ye \cite{dong_shao_ye_2012,Huang_Shao_Ye_2019,huang_ye_2005} have studied these classes (they call them $m$-sets  and $sm$-sets) in the context of disjointness between topological dynamical systems.  The family of dynamically central syndetic sets was studied in a combinatorial context by Bergelson, Hindman, and Strauss \cite{bergelson_hindman_strauss_2012} in connection to the quality of sets of return times of points to open sets under polynomial iterates of an irrational rotation.  A subfamily of dynamically syndetic sets was considered by Kennedy, Raum, and Salomon \cite{kennedy_raum_salomon_2022}; we will discuss some of their results more carefully in the context of our own below.

\subsection{The main result: a characterization of dynamically syndetic sets}
\label{sec_intro_char_of_ds_sets}

A subset of $\N$ is \emph{syndetic} if it has bounded gaps, that is, if there exists $N \in \N$ such that it has nonempty intersection with every interval of length $N$.  Syndeticity is a generalization of periodicity that is intimately connected to minimality in topological dynamics.  That dynamically syndetic sets are syndetic, for example, is a consequence of a well-known characterization of minimal actions: the system $(X,T)$ is minimal if and only if every set of the form $R(x,U)$ is syndetic.

It is not the case, however, that all syndetic sets are dynamically syndetic.  Indeed, here is an instructive example to keep in mind:
\begin{align}
    \label{eqn_intro_even_odd}
    \bigg( 2 \N \cap \bigcup_{\substack{n = 0 \\ \text{$n$ even}}}^\infty \big[ 2^{n}, 2^{n + 1}\big) \bigg) \cup \bigg( \big( 2\N + 1 \big) \cap \bigcup_{\substack{n = 0 \\ \text{$n$ odd}}}^\infty \big[ 2^{n}, 2^{n + 1}\big) \bigg).
\end{align}
That this set is syndetic is obvious.  That this set is not dynamically syndetic is less obvious.  We leave the discovery of a short argument to the intrepid reader, opting instead to show this as an application of \cref{mainthm_ds_equivalents} below.

This example raises a question that has appeared explicitly and implicitly in a number of different contexts in the literature: how do we know if a given set is dynamically syndetic? The following theorem -- the main result of our paper proved in \cref{sec_char_of_ds_and_dcs} -- answers this question by giving necessary and sufficient conditions for a set to be dynamically syndetic.

\begin{Maintheorem}
\label{mainthm_ds_equivalents}
    Let $A \subseteq \N$.  The following are equivalent.
    \begin{enumerate}
        \item
        \label{item_intro_ds_def_condition}
        The set $A$ is dynamically syndetic.

        \item
        \label{item_intro_ds_combo_condition}
        There exists a nonempty subset $B \subseteq A$ that satisfies: for all finite $F \subseteq B$, the set $\bigcap_{f \in F} (B-f)$ is syndetic.

        \item
        \label{item_intro_translate_belongs_to_sif}
        There exists $n \in \N$ for which the set $A - n$ belongs to a syndetic, idempotent filter on $\N$.
    \end{enumerate}
\end{Maintheorem}

Point \eqref{item_intro_ds_combo_condition} characterizes dynamical syndeticity in terms of a simple combinatorial property; we will discuss \eqref{item_intro_translate_belongs_to_sif} and define ``syndetic, idempotent filter'' in the next subsection. Let us use \eqref{item_intro_ds_combo_condition} to see that the set in \eqref{eqn_intro_even_odd}, call it $A$, is not dynamically syndetic.  If $B \subseteq A$ satisfies the condition in \eqref{item_intro_ds_combo_condition}, then $B$ must be syndetic.  Every syndetic subset of $A$, in particular, $B$, must contain both an even integer, $b_0$, and an odd integer, $b_1$.  It is quick to see, however, that the set $(A - b_0) \cap (A-b_1)$ is not syndetic, contradicting what we had supposed about $B$.  In addition to giving a way to check whether a set is dynamically syndetic, the combinatorial condition in \eqref{item_intro_ds_combo_condition} allows us to characterize members of the dual class, the class of dynamically thick sets, a study of which appears in a companion paper \cite{glasscock_le_2025}.

It is worth comparing point \eqref{item_intro_ds_combo_condition} in \cref{mainthm_ds_equivalents} with a recent result of Kennedy, Raum, and Salomon \cite[Thm. 8.2]{kennedy_raum_salomon_2022}. In the context of the positive integers, it is a consequence of their result that a set $A \subseteq \N$ is dynamically syndetic if and only if there exists $B \subseteq A$ such that for all finite $F_1 \subseteq B$ and $F_2 \subseteq \N \setminus B$, the set 
\begin{align}\
\label{eq:new_B'}
    \bigcap_{f_1 \in F_1} \big(B - f_1\big)  \cap  \bigcap_{f_2 \in F_2} \big((\N \setminus B) - f_2\big)
\end{align}
is syndetic.  It is not hard to see that a set $B \subseteq \N$ satisfies the condition in \eqref{eq:new_B'} if and only if its indicator function $1_B$ is a uniformly recurrent point in the full shift $(\{0,1\}^{\N}, \sigma)$; see \cref{sec_top_dynamics} for the notation and terminology.  That a set is dynamically syndetic if and only if it contains such a subset follows then from, eg., \cite[Prop. 2.3]{huang_ye_2005} (see also \cref{thm:simple_equivalent_dS}).  A main novelty of \cref{mainthm_ds_equivalents} is that the condition in \eqref{item_intro_ds_combo_condition} is substantially weaker than the one in \eqref{eq:new_B'}, in the sense that nothing is required of the complement $\N \setminus B$.
In particular, the condition in \eqref{item_intro_ds_combo_condition} does not imply that the point $1_B$ is uniformly recurrent.  This weaker condition translates to an easier-to-verify sufficient condition for dynamical syndeticity.

The following is an analogous characterization for dynamically central syndetic sets. \cref{intro_mainthm_full_chars_of_dcsyndetic} is proved in \cref{sec_char_of_ds_and_dcs}.

\begin{Maintheorem}
    \label{intro_mainthm_full_chars_of_dcsyndetic}
    Let $A \subseteq \N$.  The following are equivalent.
    \begin{enumerate}
        \item \label{intro_item_def_of_dcs}
        The set $A$ is dynamically central syndetic.

        \item \label{intro_item_dcs_combo_condition}
        There exists a subset $B \subseteq A$ that satisfies: for all finite $F \subseteq B$, the set $B \cap \bigcap_{f \in F} (B-f)$ is syndetic.

        \item \label{intro_item_member_of_sif}
        The set $A$ belongs to a syndetic, idempotent filter on $\N$.
    \end{enumerate}
\end{Maintheorem}

Straightforward combinatorial arguments give that (\ref{intro_item_dcs_combo_condition}) and (\ref{intro_item_member_of_sif}) are equivalent in Theorems \ref{mainthm_ds_equivalents} and \ref{intro_mainthm_full_chars_of_dcsyndetic}.
The novelty and difficulty of these theorems lie in the equivalence between (\ref{intro_item_def_of_dcs}) and (\ref{intro_item_member_of_sif}), a characterization dynamical syndeticity in terms of membership in a combinatorial object called a syndetic, idempotent filter.  This connection is, in fact, behind most of the other results in this paper and the companion paper \cite{glasscock_le_2025}.  We turn our attention toward these objects now.

\subsubsection{A new tool: the algebra of idempotent families}
\label{sec_intro_algebra_of_filters}

A \emph{family} is an upward-closed collection of sets.  A \emph{filter} is a family $\family$ satisfying: for all members $A_1$ and $A_2$ of $\family$, the set $A_1 \cap A_2$ belongs to $\family$.  Filters are combinatorial objects membership in which formalizes what it means for a set to be ``large'' in some way. The collection of all neighborhoods of a given point in a topological space is the quintessential example of a filter in topology.  We will see that an analogue of this filter for minimal topological dynamical systems is inextricably linked to the notion of dynamical syndeticity.

Let us briefly narrow the discussion from filters to ultrafilters.  \emph{Ultrafilters on $\N$} -- maximal filters on $\N$, whose existence is guaranteed by Zorn's lemma -- are combinatorial objects with a deep and important history in the study of topological dynamics, especially in applications to combinatorial number theory \cite{bergelson_2010}. The set $\beta \N$ of ultrafilters -- identifiable with the Stone-\Cech{} compactification of $\N$ -- is a multifaceted topological and algebraic object, the structure of which is behind many of the aforementioned applications.  For the purposes of this introduction, it suffices to know that $\beta \N$ is a compact Hausdorff space, and that there is a natural way to lift addition from $\N$ to $\beta \N$, turning $\beta \N$ into a semigroup: if $\family$ and $\familytwo$ are ultrafilters, then
\begin{align}
\label{eqn_sum_of_families_intro}
    \family + \familytwo \defeq \Big\{ A \subseteq \N \ \Big| \ \big\{ n \in \N \ \big| \ A - n \in \familytwo \big\} \in \family \Big\}
\end{align}
is an ultrafilter.  The repeated addition of the principal ultrafilter $\{A \subseteq \N \ | \ 1 \in A\}$ on the left turns $\beta \N$ into a (non-metrizable) topological dynamical system.

An ultrafilter $\family$ is called \emph{idempotent} if $\family = \family + \family$.  Idempotency of ultrafilters is closely related to proximality in topological dynamics \cite[Sec. 19.3]{hindman_strauss_book_2012} and to finite sums sets in additive combinatorics \cite[Sec. 5.2]{hindman_strauss_book_2012}.  An ultrafilter is called \emph{minimal} if it is \emph{uniformly recurrent} (see \cref{sec_top_dynamics}) under the dynamics just described.  Minimality of ultrafilters is closely related to uniform recurrence \cite[Sec. 19.3]{hindman_strauss_book_2012} and piecewise syndeticity \cite[Sec. 4.4]{hindman_strauss_book_2012}.  At the intersection of all of these ideas are Furstenberg's \emph{central sets} \cite[Def. 8.3]{furstenberg_book_1981}, combinatorially-rich, dynamically-defined subsets of $\N$ introduced to aid in the topological dynamical approach to problems in Ramsey theory. It is a result of Bergelson, Hindman, and Weiss \cite{bergelson_hindman_1990} that a subset of $\N$ is central if and only if it belongs to a minimal, idempotent ultrafilter on $\N$.

There is no obstruction to extending the familiar definition of sums of ultrafilters in \eqref{eqn_sum_of_families_intro} to arbitrary families: if $\family$ and $\familytwo$ are families, the sum $\family + \familytwo$ defined in \eqref{eqn_sum_of_families_intro} is again a family.  But while there are books written about the algebra of ultrafilters (we cite, eg., \cite{hindman_strauss_book_2012} frequently in this work), the algebra of families appears to be much less explored. For filters, the sum in \eqref{eqn_sum_of_families_intro} appears first, to our knowledge, in Berglund and Hindman \cite{berglund_hindman_1984}. We define a family $\family$ to be \emph{idempotent} if $\family \subseteq \family + \family$.  (See \cref{sec_trans_inv_and_idempotent_families} for a discussion on containment versus equality.) This definition generalizes the notion of ``idempotent filter'' that appears implicitly in Papazyan \cite{papazyan_1989} and explicitly in Krautzberger \cite{Krautzberger_2009}.

In the setting of topological dynamics, the canonical example of an idempotent filter is one of the form
\begin{align}
\label{eqn_canonical_sif}
    \family \defeq \upclose \big\{ R(x,U) \ \big| \ U \subseteq X \text{ is open and contains $x$} \big\},
\end{align}
where $x$ is a point in a dynamical system $(X,T)$ and the upward arrow indicates the upward closure of the collection of sets.  That $\family$ is a filter is clear.  To see that this is an idempotent filter, one needs only to show and interpret the following: for all open $U \subseteq X$ containing $x$,
\begin{align}
\label{eqn_inequality_to_show_idempotency}
    R(x,U) \subseteq \big\{n \in \N \ \big| \ R(x,U) - n \in \family\big\}.
\end{align}
If $n \in R(x,U)$, then $x \in T^{-n}U$, whereby $R(x,U) - n = R(x,T^{-n}U)$ is a member of $\family$.  Since $R(x,U) \in \family$, we see from the inclusion that $\{n \in \N \ | \ R(x,U) - n \in \family\} \in \family$, demonstrating that $R(x,U) \in \family + \family$.  Since every member of $\family$ contains a set of the form $R(x,U)$, this shows that $\family \subseteq \family + \family$. 

A \emph{syndetic filter} is a filter whose members are all syndetic.  If $(X,T)$ is a minimal system, the filter in \eqref{eqn_canonical_sif} is an example of a syndetic, idempotent filter.  Thus, we see that every dynamically central syndetic set belongs to a syndetic, idempotent filter. \cref{intro_mainthm_full_chars_of_dcsyndetic} gives the converse: every member of a syndetic, idempotent filter is dynamically central syndetic.  If ultrafilters are ``local'' combinatorial objects, then syndetic filters are ``global'' combinatorial objects.  One way to contextualize \cref{intro_mainthm_full_chars_of_dcsyndetic}, then, is as a global analogue to the local characterization of central sets mentioned above.

There are several other ``large,'' idempotent filters that arise naturally in the subject. Very general families of sets of return times in topological dynamics and ergodic theory, for example, are IP$^*$, idempotent filters (see \cref{sec_families_of_sets}). Our demonstration of this in a concrete example in \cref{lemma_localization_is_syndetic_idempotent_filter} is behind the proof of \cref{thm:dcps_multiple_recurrence} below.

\subsubsection{A peek into the proofs of Theorems \ref{mainthm_ds_equivalents} and \ref{intro_mainthm_full_chars_of_dcsyndetic}}
\label{sec_intro_into_the_proof}

Let us briefly discuss the proofs of Theorems \ref{mainthm_ds_equivalents} and \ref{intro_mainthm_full_chars_of_dcsyndetic}.  The full proofs, finished in \cref{sec_char_of_ds_and_dcs}, occupy a substantial portion of this paper. Since any dynamically syndetic set can be translated to become dynamically central syndetic (\cref{lemma_translates_of_dsyndetic_sets}), \cref{mainthm_ds_equivalents} follows relatively easily from \cref{intro_mainthm_full_chars_of_dcsyndetic}.

That the combinatorial condition in \cref{intro_mainthm_full_chars_of_dcsyndetic} \eqref{intro_item_dcs_combo_condition} characterizes membership in a syndetic, idempotent filter, \eqref{intro_item_member_of_sif}, is proved in \cref{sec_combo_condition_for_central_syndetic}.  The proof is purely combinatorial and highlights some of the family algebra discussed in the previous section.  Readers familiar with the characterization of IP sets (see \cref{sec_families_of_sets}) as members of idempotent ultrafilters will immediately recognize the ideas at play.

That dynamically central syndetic sets belong to a syndetic, idempotent filter (that \eqref{intro_item_def_of_dcs} implies \eqref{intro_item_member_of_sif}) is simple and was discussed in the previous section. The challenge lies in the converse, whose proof occupies most of \cref{sec_central_and_dy_central_sets}. Since this is the most technical part of the paper, let us briefly describe the main idea.  A more detailed outline is given in \cref{sec_main_idea_of_shift_punch}.

We begin with a set $A \subseteq \N$ that satisfies the combinatorial condition in \eqref{intro_item_dcs_combo_condition}.  Were the point $1_{A \cup \{0\}}$ uniformly recurrent in the full symbolic shift $(\{0,1\}^{\N \cup \{0\}},\sigma)$, it would follow from \cref{thm:simple_equivalent_dcS} that $A$ is dynamically central syndetic.  It is not true in general, however, that $1_{A \cup \{0\}}$ is uniformly recurrent.  The main idea is to modify the symbolic shift by composing it with a ``punch'' map to form a related system that we call ``shift-punch.''  Under the shift-punch dynamics, the point $1_{A \cup \{0\}}$ is uniformly recurrent.  It is this uniform recurrence, then, that allows us to find a subset of $A$ that is dynamically central syndetic.

\subsection{Applications to sets of pointwise recurrence}

A set $A \subseteq \N$ is a \emph{set of (minimal, topological) pointwise recurrence} if for all minimal systems $(X,T)$, all points $x \in X$, and all open $U \subseteq X$ containing $x$, there exists $n \in A$ such that $T^n x \in U$.
The family of sets of pointwise recurrence is dual to the family of dynamically central syndetic sets, in the sense described in \cref{sec_dual_families}.
As an application of our main results, we answer two questions of Host, Kra, and Maass concerning sets of pointwise recurrence which we now describe.

\subsubsection{Sets of pointwise recurrence are sets of polynomial recurrence}

A set $A \subseteq \N$ is called a \emph{set of multiple topological recurrence} if for all $k \in \N$, all minimal systems $(X, T)$, and all nonempty, open sets $U \subseteq X$, there exists $n \in A$ such that
\[
    U \cap T^{-n} U \cap \cdots \cap T^{-kn} U \neq \emptyset.
\]
The set $A$ is called a \emph{set of (single) topological recurrence} if it satisfies the condition above for $k=1$.
Since Furstenberg and Weiss \cite{Furstenberg_Weiss-topological-dynamics} proved that $\N$ is a set of multiple topological recurrence, the problem of determining necessary and sufficient conditions for a subset of $\N$ to be such has attracted much attention.
It is immediate from the definitions that a set of pointwise recurrence is a set of single topological recurrence.
That there are sets of single topological recurrence that are not sets of multiple topological recurrence was shown by Furstenberg \cite{furstenberg_book_1981} (see also \cite{frantzikinakis_lesigne_wierdl_2006}).
This is the context for the following question.

\begin{question}[{Host-Kra-Maass \cite[Question 2.14]{HostKraMaass2016}}]
\label{quest_hkm_2_14}
    Is a set of pointwise recurrence a set of multiple topological recurrence?
\end{question}

Our next theorem gives a strong positive answer to \cref{quest_hkm_2_14}: sets of pointwise recurrence are sets of polynomial multiple measurable recurrence for commuting transformations.  A set $A \subseteq \N$ is a \emph{set of polynomial multiple measurable recurrence for commuting transformations} if for all $k, \ell \in \N$, all commuting measure preserving systems $(X, \mu, T_1, \ldots, T_k)$ (see \cref{sec_dcps_and_set_recurrence} for the precise definition), all $E \subseteq X$ with $\mu(E) > 0$, and all polynomials $p_{i,j} \in \Q[x]$ with $p_{i, j}(0) = 0$ and $p_{i, j}(\N_0) \subseteq \N_0$, $1 \leq i \leq k$, $1 \leq j \leq \ell$, there exists $n \in A$ such that
\[
    \mu\Big( \big(T_1^{p_{1,1}(n)} T_2^{p_{2,1}(n)} \cdots T_k^{p_{k,1}(n)} \big)^{-1} E \cap \cdots \cap \big(T_1^{p_{1,\ell}(n)} T_2^{p_{2,\ell}(n)} \cdots T_k^{p_{k,\ell}(n)} \big)^{-1} E \Big) > 0.
\]

Because any minimal topological dynamical system admits an invariant measure of full support, sets of polynomial multiple measurable recurrence for commuting transformations are sets of multiple topological recurrence.

\begin{Maintheorem}
\label{thm:dcps_multiple_recurrence}
    Sets of pointwise recurrence are sets of polynomial multiple measurable recurrence for commuting transformations.
\end{Maintheorem}

The proof of \cref{thm:dcps_multiple_recurrence} in \cref{sec_dcps_and_set_recurrence} appeals to the IP polynomial \Szemeredi{} theorem of Bergelson and McCutcheon \cite{Bergelson-McCutcheon-IPpolynomialSzemeredi} to show that the family of times of polynomial returns in ergodic theory is an $\IP^*$ (and, hence, syndetic), idempotent filter. It follows then immediately from \cref{intro_mainthm_full_chars_of_dcsyndetic} that sets of pointwise recurrence have nonempty intersection with every member of this filter, proving \cref{thm:dcps_multiple_recurrence}.

Combining \cref{thm:dcps_multiple_recurrence} with the fact that sets of pointwise recurrence are piecewise syndetic (\cref{ex_ps_star_is_dcs}), we show in  \cref{thm:Brauer_and_dcPS} that every set of pointwise recurrence contains Brauer-type polynomial configurations.  The full extent of the combinatorial richness of sets of pointwise recurrence remains an interesting open question.

\subsubsection{The family of sets of pointwise recurrence is not partition regular}
\label{sec_intro_pwrec_not_pr}

It is a consequence of \cref{lemma_localization_is_syndetic_idempotent_filter} -- a result on the way toward proving \cref{thm:dcps_multiple_recurrence} -- that the family of sets of polynomial multiple measurable recurrence for commuting transformations is partition regular: if $A$ is such a set and $A = B \cup C$, then at least one of $B$ or $C$ is such.  For sets of single and multiple topological recurrence and their measurable counterparts, this is a well known and often-exploited fact; see, eg., \cite[Proposition 6.2]{HostKraMaass2016} and \cite[Lemma 3.3]{Furstenberg-poincare_recurrence}.
In this context, it is natural to ask whether the family of sets of pointwise recurrence is also partition regular.

\begin{question}[{Host-Kra-Maass \cite[Question 6.5]{HostKraMaass2016}}]
\label{quest_is_pw_rec_partition_reg}
    Is the family of sets of pointwise recurrence partition regular?
\end{question}

We give a strong negative answer to \cref{quest_is_pw_rec_partition_reg}, showing that even $\N$ can be partitioned into two sets, neither of which is a set of pointwise recurrence. This result demonstrates a sharp contrast between the nature of topological and measurable set recurrence and pointwise recurrence.

\begin{Maintheorem}
\label{mainthm_partition_of_dcs_sets}
    Every dynamically central syndetic set can be partitioned into infinitely many, pairwise disjoint dynamically central syndetic sets. As a result, every set of pointwise recurrence -- in particular, the set of positive integers $\N$ -- can be partitioned into two sets, neither of which is a set of pointwise recurrence.
\end{Maintheorem}

The proof of \cref{mainthm_partition_of_dcs_sets}, given in \cref{sec_partitioning_dcs_sets}, hinges on the symbolic characterization of dynamically syndetic sets given in \cref{thm:simple_equivalent_dcS}.  This characterization is, in turn, one of the essential ingredients in the proof of \cref{intro_mainthm_full_chars_of_dcsyndetic}.

\subsection{A companion paper on sets of pointwise recurrence and dynamical thickness}

Dual to the families of dynamically syndetic and dynamically central syndetic sets are the families of dynamically thick sets and sets of pointwise recurrence.  We study these families -- and the families obtained by intersecting dynamically syndetic and dynamically thick sets -- in depth in a companion paper \cite{glasscock_le_2025}.  The main results in that paper rely critically on the main results here, Theorems \ref{mainthm_ds_equivalents} and \ref{intro_mainthm_full_chars_of_dcsyndetic}.  In contrast, although we make frequent reference to that paper, the present paper is logically independent and does not rely on any of the results from \cite{glasscock_le_2025}.

\subsection{Organization of the paper}

Notation, terminology, and preliminaries are given in \cref{sec_prelims}.  We give in \cref{sec_dsyndetic} some first results on dynamically syndetic sets, then characterize in \cref{sec_central_and_dy_central_sets} dynamically central syndetic sets as members of syndetic, idempotent filters.  Theorems \ref{mainthm_ds_equivalents}, \ref{intro_mainthm_full_chars_of_dcsyndetic}, \ref{thm:dcps_multiple_recurrence}, and \ref{mainthm_partition_of_dcs_sets} are proven in \cref{sec_dps_sets}, and we conclude with some open problems in \cref{sec:open_questions}.

\subsection{Acknowledgments}

Thanks goes to Andy Zucker and Josh Frisch for helpful discussions at the 2024 Southeastern Logic Symposium Conference and for the references they provided.  Thanks also goes to Neil Hindman for correspondence regarding translates of central sets.  The authors are indebted to Eli Glasner, John Johnson, and Mauro Di Nasso for their assistance in correcting and improving several of the historical references and claims made in a first draft of the paper. Last but not least, thanks go to Andreas Koutsogiannis, Joel Moreira, Florian Richter, and Donald Robertson, with whom the authors were working when a number of the ideas in this paper sharpened.

\section{Preliminaries}
\label{sec_prelims}

The set of integers, non-negative integers, and positive integers are denoted by $\Z$, $\N_0$, and $\N$, respectively.  The power set of $\N$ is denoted $\mathcal{P}(\N)$.  For $N \in \N_0$, we define $[N]$ to be $\{0, 1, \ldots, N - 1\}$.

For $A \subseteq \N$ and $n \in \N$, we define
\begin{align*}
    A+n &\defeq \{a + n \ | \ a \in A \}, & nA & \defeq \{ na \ | \ a \in A\}, \\
    A-n &\defeq \{m \in \N \ | \ m + n \in A \}, & A/n & \defeq \{ m \in \N \ | \ mn \in A\}.
\end{align*}
For $A, B \subseteq \N$, we define
\begin{align*}
    A+B &\defeq \bigcup_{b \in B} (A + b), & A-B &\defeq \bigcup_{b \in B} (A-b).
\end{align*}
We follow the convention that the empty union of subsets of $\N$ is empty while the empty intersection of subsets of $\N$ is equal to $\N$. Note that all set operations on subsets of $\N$ result in subsets of $\N$.

\subsection{Families of positive integers}
\label{sec_combinatorics}

Much of the algebra described in this subsection goes back to work of Choquet \cite{choquet_1947} and Schmidt \cite{schmidt_1952,schmidt_1953}.  A good modern reference for this material can be found in \cite[Section 2]{christopherson_johnson_2022}.  Many of the results here are generalized to arbitrary semigroups in \cite{Burns_Davenport_Frankson_Griffin_Johnson_Kebe_syndetic}.

Let $\class$ be a collection of subsets of $\N$.  The \emph{upward closure} of $\class$ is
\[\upclose \class \defeq \big\{ B \subseteq \N \ \big| \ \exists A \in \class, \ B \supseteq A \big\}.\]
The set $\class$ is \emph{upward-closed} if $\class = \upclose \class$.  We will call an upward-closed collection of subsets of $\N$ a \emph{family}.  Note that both $\emptyset$ and $\mathcal{P}(\N)$ are families.  A family is \emph{proper} if it is neither $\emptyset$ nor $\mathcal{P}(\N)$, that is, it is both nonempty and does not contain the empty set.

\subsubsection{Dual families and intersections}
\label{sec_dual_families}

The \emph{dual} of a family $\family$ of subsets of $\N$ is
\[\family^* \defeq \big\{ B \subseteq \N \ \big | \ \forall A \in \family, \ B \cap A \neq \emptyset \big\},\]
the collection of all those sets which have nonempty intersection with all elements of $\family$.  The following facts are quick to check and will be used without mention:
\begin{enumerate}
    \item $(\family^*)^* = \family$, whereby it makes sense to say that $\family$ and $\family^*$ are \emph{dual};
    \item the empty family $\emptyset$ and $\mathcal{P}(\N)$ are dual, while the dual of a proper family is a proper family;
    \item $A \in \family$ if and only if $\N \setminus A \not\in \family^*$;
    \item $\family \subseteq \familytwo$ if and only if $\familytwo^* \subseteq \family^*$.
\end{enumerate}

\subsubsection{Filters, partition regularity, and ultrafilters}
\label{sec_filter_pr_and_ufs}

There are two complementary notions of largeness for families that feature prominently in this topic: that of being partition regular and that of being a filter.  A family $\family$ is \emph{partition regular} if for all $A \in \family$ and all finite partitions $A = A_1 \cup \cdots \cup A_k$, some piece $A_i$ belongs to $\family$.  The family $\family$ is a \emph{filter} if for all $A_1, A_2 \in \family$, the set $A_1 \cap A_2 \in \family$.  These notions are dual in the sense of fact \eqref{item_pr_filter_duality} below.

If $P$ is a property of families, we call a family that has property $P$ and that is also a filter a \emph{$P$ filter}. For example, we will consider proper filters, partition regular filters, syndetic filters, and so on. Note, for example, that $\emptyset$ and $\mathcal{P}(\N)$ are both partition regular filters.

The following facts are quick to check:
\begin{enumerate}
    \item \label{item_pr_filter_duality} a family $\family$ is partition regular if and only if the dual family $\family^*$ is a filter;
    \item if $\family$ is a proper filter, then $\family \subseteq \family^*$.\\
\end{enumerate}

\subsubsection{Translation algebra}
\label{sec_family_algebra}

For $A \subseteq \N$ and families $\family$ and $\familytwo$ of subsets of $\N$, we define
\begin{align*}
    A - \family &\defeq \big\{ n \in \N \ \big| \ A-n \in \family \big\},\\
    \familyone + \familytwo &\defeq \big\{ B \subseteq \N \ \big| \ B - \familytwo \in \family \big\}.
\end{align*}
Thus, we see that $n \in A - \family$ if and only if $A - n \in \family$ and that $B \in \family + \familytwo$ if and only $B - \familytwo \in \family$.  It is simple to check that $\family + \familytwo$ is a family of subsets of $\N$.

The definition of $\family + \familytwo$ appears for filters in \cite{berglund_hindman_1984,papazyan_1989} and, when $\family$ and $\familytwo$ are ultrafilters, agrees with the usual definition of sums of ultrafilters (cf. \cite[Thm. 4.12]{hindman_strauss_book_2012}).  Family sums will appear throughout the paper.  It is useful to note that, loosely speaking, the containment $\family \subseteq \familytwo + \familythree$ means: for all $A \in \family$, there are $\familytwo$ many positive integers $n$ for which the set $A-n$ belongs to $\familythree$.

The following lemma records a number of algebraic facts that will be useful later on.

\begin{lemma}
    \label{lemma_filter_algebra}
    
    Let $n \in \N$, $A, B \subseteq \N$, and $\filter, \filtertwo$ be families.
    \begin{enumerate}
        \item
        \label{item_family_prop_one_new}
        If $\family$ is a filter, then $(A - \filter) \cap (B - \filter) = (A \cap B) - \filter$.

        \item \label{item_family_prop_three}
        $(A -n) - \filter = (A - \filter) - n$.

        \item \label{item_family_prop_four}
        $A - (\filter + \filtertwo) = (A - \filtertwo) - \filter$.

        \item \label{item_sum_dual}
        $(\familyone + \familytwo)^* = \familyone^* + \familytwo^*$.
    \end{enumerate}
\end{lemma}

\begin{proof}
    \eqref{item_family_prop_one_new} \  Suppose $\family$ is a filter. To see that $(A - \filter) \cap (B - \filter) \subseteq (A \cap B) - \filter$, let $n \in (A - \filter) \cap (B - \filter)$.  Since $A - n \in \filter$ and $B - n \in \filter$, we have that $(A \cap B) - n = (A-n) \cap (B-n) \in \filter$.  Thus, $n \in (A \cap B) - \filter$. To see that $(A \cap B) - \family \subseteq (A- \family) \cap (B-\family)$, let $n \in (A \cap B) - \family$ so that $(A - n) \cap (B-n) \in \family$.  Since $\family$ is upward closed, we have that $A -n \in \family$ and $B-n \in \family$.  Therefore, $n \in (A - \family) \cap (B-\family)$, as desired.

    \eqref{item_family_prop_three} \ Note that $m \in (A -n) - \filter$ if and only if $A-n-m \in \filter$, which happens if and only if $A-(m+n) \in \filter$.  We see that this happens if and only if $m+n \in A-\filter$, that is, if and only if $m \in (A-\filter) - n$, as desired.

    \eqref{item_family_prop_four} \ We see that $n \in A - (\filter + \filtertwo)$ if and only if $A-n \in \filter + \filtertwo$, which happens if and only if $(A-n) - \filtertwo \in \filter$.  By \eqref{item_family_prop_three}, this happens if and only if $(A-\filtertwo) - n \in \filter$, which happens if and only if $n \in (A - \filtertwo) - \filter$, as desired.
    
    \eqref{item_sum_dual} \ We will show first that $\familyone^* + \familytwo^* \subseteq (\familyone + \familytwo)^*$.  Suppose $A \in \familyone^* + \familytwo^*$, and let $B \in \familyone + \familytwo$.  Since $A - \familytwo^* \in \familyone^*$ and $B - \familytwo \in \familyone$, we have that there exists $n \in (A - \familytwo^*) \cap (B - \familytwo)$.  We see that $A - n \in \familytwo^*$ and $B - n \in \familytwo$.  It follows that $(A-n) \cap (B-n) \neq \emptyset$, whereby $A \cap B \neq \emptyset$, as desired.

    Now we can write
    \[\familyone^* + \familytwo^* \subseteq (\familyone + \familytwo)^* = \big((\familyone^*)^* + (\familytwo^*)^* \big)^* \subseteq \big((\familyone^* + \familytwo^*)^* \big)^* = \familyone^* + \familytwo^*,\]
    where the first containment follows from the previous paragraph and the second containment follows from the previous paragraph applied to the dual classes.  This shows the desired equality.
\end{proof}

\subsubsection{Translation-invariant and idempotent families}
\label{sec_trans_inv_and_idempotent_families}

For $m \in \N$ and a family $\family$ of subsets of $\N$, we define
\[\family - m \defeq \big\{ B - m \ \big| \ B \in \family \big\}.\]
The family $\family$ is \emph{translation invariant} if $\family - n \subseteq \family$ for all $n \in \N$. In the algebra of the previous section, this is equivalent to saying that $\family \subseteq \{\N\} + \family$.  The family $\filter$ is \emph{idempotent} if $\filter \subseteq \filter + \filter$, equivalently, if for all $A \in \filter$, $A - \filter \in \filter$. (We comment on why the set inclusion in the definition is not an equality in \cref{rmk_why_set_inclusion_not_equality} below.) Note that all translation-invariant families are idempotent.  We will give some less trivial examples of idempotent families in \cref{sec_families_of_sets}.

\begin{remark}
    \label{rmk_showing_idempotency}
    The following situation arises several times in this paper.  A family $\family$ is defined as the upward closure of a collection of special subsets of $\N$.  To show that $\family$ is idempotent, it suffices to show that for all special sets $A \subseteq \N$, the set $A - \filter$ belongs to $\family$.  Indeed, if $B \in \filter$, then $B$ contains some special set $A$.  If we show that $A-\filter \in \filter$, then since $A- \filter \subseteq B - \filter$, we have that $B - \filter \in \filter$, as desired.
\end{remark}

\begin{remark}
\label{rmk_why_set_inclusion_not_equality}
    The definition of idempotent family we give in this section matches the one for filters given implicitly in \cite{papazyan_1989} and explicitly in \cite{Krautzberger_2009}.  It may seem more natural to require equality in the definition: $\filter = \filter + \filter$.  While certainly some families do satisfy this, equality appears to be rather rare.  Here is an example of a dynamically simple idempotent filter $\filter$ which satisfies $\filter \subsetneq \filter + \filter$.  See \cref{sec_top_dynamics} for the notation.
    
    Fix $\alpha \in \R \setminus \Q$, and consider the irrational rotation $(\R / \Z, T_{\alpha}: x \mapsto x + \alpha)$.  Put
    \[\filter \defeq \upclose \big\{ R(0, B_\eps(0)) \ \big| \ \eps > 0\big\}.\]
    The family $\filter$ is an idempotent filter.  This can be easily checked (see, eg., the proof of the ``only if'' statement of \cref{thm_dcs_iff_cs}).

    Let $(k_i)_{i=1}^\infty \subseteq \N$ be such that the only limit point of $i \mapsto k_i \alpha$ is $0$. For $n \in \N$, the quantity
    \[\eps_n \defeq \min \big\{ \|k_i \alpha - n \alpha \|/2 \ \big| \ i \in \N, \ k_i > n \big\},\]
    where $\| \cdot \|: \R / \Z \to [0,1/2)$ denotes the Euclidean distance to $\Z$, is non-zero.
    We claim that the set
    \[A \defeq \bigcup_{n = 1}^\infty \big( R(0, B_{\eps_n}(0) ) + n \big)\]
    belongs to $\filter + \filter$ but not to $\filter$.  Indeed, for all $n \in \N$, the set $A - n$ contains $R(0, B_{\eps_n}(0) ) \in \filter$, and so $A \in \{\N\} + \filter \subseteq \filter + \filter$.

    To see that $A \not\in \filter$, we will show that for all $\delta > 0$, we have $R(0,B_{\delta}(0)) \not\subseteq A$.  Let $\delta > 0$, and let $i \in \N$ such that $k_i \in R(0,B_{\delta}(0))$.  Now, for all $n < k_i$, we have that $\|k_i \alpha - n \alpha\| \geq 2\eps_n$, whereby $k_i \notin R(0,B_{\eps_n}(0)) + n$.  For all $n \geq k_i$, we have that $k_i \notin R(0,B_{\eps_n}(0)) + n$ since $R(0,B_{\eps_n}(0)) + n \subseteq \{n+1, n+2, \ldots\}$.  Therefore, we have that $k_i \not\in A$, as was to be shown.
\end{remark}

\subsection{Topology and dynamics}

For a set $U$ in a topological space $X$, we denote by $\overline U$, $U^\circ$, and $\partial U \defeq \overline{U} \setminus U^\circ$ the closure, interior, and boundary of $U$, respectively.  In a metric space $(X,d)$, we denote by $B_\eps(x)$ the open ball centered at a point $x \in X$ with radius $\eps > 0$.

\subsubsection{Topological dynamics}
\label{sec_top_dynamics}

A \emph{(topological dynamical) system} $(X,T)$ is a nonempty, compact metric space $(X,d_X)$ together with a continuous self-map $T: X \to X$.  A \emph{subsystem of $(X,T)$} is a system of the form $(Y,T)$, where $Y \subseteq X$ is nonempty, closed (equivalently, compact), and \emph{$T$-invariant}, meaning $TY \subseteq Y$.  A system $(X,T)$ whose only subsystem is itself is called \emph{minimal}.  All systems in this paper will be considered as actions of the semigroup $(\N, +)$ by continuous maps on a compact metric space.  Thus, even in the event that $T$ is a homeomorphism (in which case, the system $(X,T)$ is called \emph{invertible}), the emphasis will be on non-negative iterates of the map $T$.

Given a system $(X,T)$, a point $x \in X$, and a set $A \subseteq \N$, we define
\[T^A x \defeq \big\{ T^a x \ \big| \ a \in A \big\}.\]
The \emph{orbit of $x$} is the set $T^{\N} x$ and the \emph{orbit closure of $x$} is the set $\overline{T^{\N}x}$. Note that $\overline{T^{\N}x}$ is a nonempty, closed, $T$-invariant subset of $X$, whereby $(\overline{T^{\N}x}, T)$ is a subsystem of $(X,T)$.  Thus, we see that $(X,T)$ is minimal if and only if every point $x \in X$ has a dense orbit.

Recall from the introduction that given a system $(X,T)$, a set $U \subseteq X$, and a point $x \in X$, we write
\begin{align*}
    R(x, U) &\defeq \big\{n \in \N \ \big | \ T^n x \in U \big\}
\end{align*}
for the times at which the point $x$ visits the set $U$. To emphasize the transformation $T$, we sometimes write this set as $R_T(x,U)$. It is a simple algebraic fact that we will use without mention that for $n \in \N_0$,
\[R(x,U) - n = R(T^n x, U) = R(x, T^{-n}U).\]

A set $A \subseteq \N$ is \emph{syndetic} if there exists $N \in \N$ such that $A \cup (A-1) \cup \cdots \cup (A-N) = \N$.
Syndetic sets -- discussed in more context in \cref{sec_families_of_sets} -- are intimately linked to minimal dynamics.  Indeed, let $(X,T)$ be a system.  Both of the facts below are not hard to show following the results in \cite[Ch. 1, Sec. 4]{furstenberg_book_1981}.
\begin{enumerate}
    \item Let $x \in X$.  The subsystem $(\overline{T^{\N_0}x}, T)$ is minimal if and only if for all open $U \subseteq X$ containing $x$, the set $R(x,U)$ is syndetic. In this case, the point $x$ is said to be \emph{uniformly recurrent.}

    \item The system $(X,T)$ is minimal if and only if for all $x \in X$ and all nonempty, open $U \subseteq X$, the set $R(x,U)$ is syndetic.
\end{enumerate}
We will use these facts several times in this paper.

\begin{lemma}
    \label{lemma_isolated_points_and_minimal_systems}
    Let $(X,T)$ be a minimal system.  There exists an isolated point in $X$ if and only if $X$ is finite and the system $(X,T)$ is periodic.
\end{lemma}

\begin{proof}
    If $x \in X$ is isolated, the set $U \defeq \{x\}$ is open.  Since $(X,T)$ is minimal, the set $R(x,U)$ is syndetic. In particular, there exists $n \in \N$ such that $T^n x = x$, whereby $x$ is a periodic point.  Since $(X,T)$ is minimal, the set $X$ is equal to the (finite) orbit of $x$ and $(X,T)$ is periodic.  The converse is trivial.
\end{proof}

\subsubsection{Symbolic space and full shift}
\label{sec:symbolic_prelim}

We denote the space of 0--1 valued functions on $\N$ by $\{0,1\}^{\N}$. Given the product topology, this is a compact Hausdorff space.  There are many equivalent metrics that generate the product topology on $\{0,1\}^{\N}$.  Instead of specifying one of these metrics explicitly, it suffices for our purposes to fix one and note that two elements of $\{0,1\}^{\N}$ are close if and only if they agree as functions on a long initial interval $\{1, \ldots, N\}$ of $\N$.

Let $\omega \in \{0,1\}^{\N}$ and $i \in \N$.  We write $\omega(i)$ for the value of the function $\omega$ at $i$.  Given $k \in \{0,1\}$, the \emph{cylinder set} of all those functions in $\{0,1\}^{\N}$ that evaluate to $k$ at $i$ is denoted by $[k]_i$.  This is easily checked to be a clopen subset of $\{0,1\}^{\N}$.  The \emph{support} of $\omega$ is
\[\supp (\omega) \defeq \big\{ n \in \N \ \big| \ \omega(n) = 1 \big\}.\]

We define the (left) shift $\sigma: \{0,1\}^{\N} \to \{0,1\}^{\N}$ at the function $\omega$ by the rule $(\sigma \omega)(i) = \omega(i+1)$, $i \in \N$.  The shift is easily seen to be a continuous self-map of $\{0,1\}^{\N}$.  The \emph{full shift} is the system $(\{0,1\}^{\N}, \sigma)$.

It will frequently be useful to consider functions on $\N_0$ instead of $\N$.  Everything written above applies equally well to the space $\{0,1\}^{\N_0}$ and the \emph{full shift} $(\{0,1\}^{\N_0}, \sigma)$.

\subsection{Families of sets from combinatorics and dynamics}
\label{sec_families_of_sets}

We outline below those families of subsets of $\N$ that are important in this work.  The following nomenclature will be particularly convenient.  Let $\familytwo$ be a family. A set $A \subseteq \N$ is called a \emph{$\familytwo$ set} if it is a member of $\familytwo$, and a family $\familyone$ is called a \emph{$\familytwo$ family} if all members of $\familyone$ are $\familytwo$ sets, ie., $\familyone \subseteq \familytwo$. We will frequently consider in this work, for example, syndetic filters, ie., filters whose members are all syndetic.

A set $A \subseteq \N$ is \dots
\begin{enumerate}
    \item \dots \emph{syndetic} if there exists $N \in \N$ such that
    \[
        A \cup (A-1) \cup \cdots \cup (A-N) = \N;
    \]

    \item \dots \emph{thick} if for all finite $F \subseteq \N$, there exists $n \in \N$ such that $F + n \subseteq A$;

    \item \dots \emph{piecewise syndetic} if there exists $N \in \N$ such that the set
    \[A \cup (A-1) \cup \cdots \cup (A-N)\]
    is thick.
\end{enumerate}
We denote by $\syndetic$, $\thick$, and $\PS$ the families of syndetic, thick, and piecewise syndetic subsets of $\N$, respectively. All three families are translation invariant. It is well-known and not difficult to show that the families of syndetic and thick sets are dual and that
\[\PS = \big\{ A \cap B \ \big| \ A \in \syndetic, \ B \in \thick \big\}.\]
The family $\PS$ is partition regular \cite[Lemma 1]{brown_1971} (see also \cite[Prop. 2.5(h)]{christopherson_johnson_2022}) and so its dual, $\PS^*$, is a filter. It is not hard to show that a set $A$ belongs to $\PS^*$ if and only if for all finite $F \subseteq \N$, there exists a syndetic set $S \subseteq \N$ such that $F + S \subseteq A$.  Thus, members of the filter $\PS^*$ are called \emph{thickly syndetic} sets.

The following lemma gives a useful characterization of the filter of thickly syndetic sets as the largest syndetic, translation-invariant filter.

\begin{lemma}
\label{lemma_condition_on_subfamily_of_ps_star}
    The family $\PS^*$ is a syndetic, translation-invariant filter.  If $\family$ is a syndetic, translation-invariant filter, then $\family \subseteq \PS^*$.
\end{lemma}

\begin{proof}
    That $\PS^*$ is a syndetic, translation-invariant filter was discussed above.
    
    Let $\family$ be a syndetic, translation-invariant filter, and let $A \in \family$.  Since $\family$ is translation invariant, for all $N \in \N$, the sets $A-1$, \dots, $A-N$ all belong to $\family$. Since $\family$ is a syndetic filter, the set $(A - 1) \cap \cdots \cap (A - N)$ belongs to $\family$ and is syndetic.  Since
    \[\{1, \ldots, N\} + \big( (A - 1) \cap \cdots \cap (A - N) \big) \subseteq A,\]
    we see that $A$ is thickly syndetic and hence is a member of $\PS^*$.
\end{proof}

The family of infinite, finite sums sets will also appear in this work.
A set $A \subseteq \N$ is called an \emph{IP set} if there exists a sequence $(x_i)_{i=1}^\infty \subseteq \N$ such that
\begin{align}
\label{eqn_def_of_fs_set}
    \FS(x_i)_{i=1}^\infty \defeq \left\{ \sum_{f \in F} x_f \ \middle | \ F \subseteq \N \text{ is finite and nonempty} \right\} \subseteq A.
\end{align}
Here ``FS'' stands for ``finite sums.''  It is a consequence of Hindman's theorem \cite{hindman_1974} that the family $\IP$ is partition regular (see \cite[Lemma 2.1]{bergelson_hindman_1993}), and it is an easy exercise that thick sets are IP.  Thus, the dual family, $\IP^*$, is a syndetic filter.  It is a short exercise to check that $\IP$ is an example of an idempotent family that is not translation invariant.

\section{First results on dynamically syndetic sets}
\label{sec_dsyndetic}

The families of dynamically syndetic and dynamically central syndetic subsets of $\N$ -- defined at the top of \cref{sec:intro} -- will be denoted henceforth by $\dsyndetic$ and $\dcsyndetic$, respectively.  It is clear from the definitions and the facts in \cref{sec_top_dynamics} that
\[\dcsyndetic \subseteq \dsyndetic \subseteq \syndetic.\]
The set of odd, positive integers and the set in \eqref{eqn_intro_even_odd} from the introduction demonstrate that each of these inclusions is proper. (That $2\N-1 \not\in \dcsyndetic$ follows from, eg., \cref{lemma_dilates_of_dsyndetic_sets}.)  Here is a short list of examples of dynamically (central) syndetic sets that may be helpful to have in mind.  Let $\alpha \in \R \setminus \Q$.
\begin{enumerate}
    \item \label{item:irrational_rotation_dcS_definition}
    The irrational rotation $(\R / \Z, x \mapsto x + \alpha)$, where $\alpha$ is interpreted as an element of $\R / \Z$, is a minimal system.  For all $x \in \R / \Z$ and all nonempty, open $U \subseteq \R / \Z$, the set
    \[R(x,U) = \big\{ n \in \N \ \big| \  x + n \alpha  \in U \big\}\]
    is dynamically syndetic and, if $x \in U$, dynamically central syndetic.  More generally, supersets of Bohr (Bohr$_0$) sets -- those arising in this way from compact group rotations -- are dynamically (central) syndetic.

    \item When $\alpha > 1$, the sequence $(\lfloor n \alpha \rfloor)_{n \in \N}$, where $\lfloor \, \cdot \, \rfloor: \R \to \Z$ denotes the integer part function, is called a \emph{Beatty sequence}.  It is not hard to show that the image of this sequence is equal to the set
    \[
        \big\{ n \in \N \ \big| \ \big\{ n / \alpha \} \in \big( (\alpha-1) / \alpha, 1 \big) \big\},
    \]
    where $\{ \, \cdot \, \}: \R \to [0,1)$ denotes the fractional part function. Thus, Beatty sequences yield Bohr, and hence dynamically syndetic, sets.  What is less clear is that, in fact, Beatty sequences yield dynamically central syndetic sets.  We will see this as a consequence of \cref{thm:simple_equivalent_dcS}.

    \item Using the notation from \eqref{item:irrational_rotation_dcS_definition}, the set
    \[\big\{ n \in \N \ \big| \  n^2 \alpha  \in U \big\}\]
    is a dynamically syndetic set.  It is, in fact, an example of a nil-Bohr set, one that arises from a rotation on a compact homogeneous space of a nilpotent Lie group.  Nil-Bohr (nil-Bohr$_0$) sets are dynamically (central) syndetic \cite{Host-Kra_nilbohr}.
    
    \item \label{item_dcs_two_adic_example} 
    The function $\nu_2: \N \to \N_0$ defined by $\nu(n) \defeq \max \{ k \in \N_0 \ | \ 2^k \text{ divides } n \}$ is called the \emph{$2$-adic valuation}.
    The set
    \[
        A \defeq \big\{ n \in \N \ \big | \ \nu_2(n) \text{ is even} \big \}
    \]
    is a dynamically syndetic set.  To see this, one can show that the point $1_A$ is uniformly recurrent in the full shift $(\{0,1\}^{\N_0},\sigma)$, whereby $A = R_\sigma(1_A,[1]_0)$ is dynamically syndetic.  In fact, the point $1_{A \cup \{0\}} \in [1]_0$ is uniformly recurrent, and so $A$ is dynamically central syndetic.  By \cref{lemma_dilates_of_dsyndetic_sets} and the fact that $\N \setminus A = A / 2$, the same can be said about the set of positive integers with odd 2-adic valuation.
    
    \item The \emph{Chacon word} $0010001010010 0010001010010 1 0010001010010\ldots$ is the unique, non-constant word in the symbols $0$ and $1$ that is fixed under the substitutions $0 \mapsto 0010$ and $1 \mapsto 1$.  The set of locations of the 1's in the Chacon word is a dynamically syndetic set, for the same reason as in the previous point: the word, interpreted as a point in the full shift $(\{0,1\}^{\N_0},\sigma)$, is known to be uniformly recurrent \cite[Sec. 5.6.6]{devries_2014}.\\
\end{enumerate}

Many of the results in this section demonstrate the robustness of the family of dynamically (central) syndetic sets under changes to the combinatorial, topological, or dynamical requirements in its definition. The family of dynamically (central) syndetic subsets of $\N$ remains unchanged if \dots
\begin{enumerate}
    \item \dots \ the map $T$ is required to be a homeomorphism, that is, if the system $(X,T)$ is required to be invertible.  This is shown in \cref{lemma_may_assume_invertible} below.

    \item \dots \ the space $X$ is allowed to be non-metrizable. This is a consequence of \cite[Lemma 3.1]{glasscock_le_2025}.
    
    \item \dots \ the space $X$ is required to be zero dimensional and the set $U$ is required to be clopen. This is a consequence of \cref{thm:simple_equivalent_dcS}.

    \item \dots \ (for dynamically central syndetic sets) the point $x$ is allowed to be in $\overline{U}$ instead of being required to be in $U$.  Sets of this form are called \emph{very strongly central} in \cite[Def. 2.10]{bergelson_hindman_strauss_2012}, where a number of equivalent characterizations in terms of ultrafilters are given.  That the families of very strongly central sets and dynamically central syndetic sets are the same is shown in \cref{thm:simple_equivalent_dcS}.

    \item \dots \ sets of return times $R(x,U)$ are considered up to equivalence on thickly syndetic sets ($\PS^*$ sets). This is a consequence of \cite[Theorem 6.12]{glasscock_le_2025}: a set $A \subseteq \N$ is dynamically (central) syndetic if and only if for all thickly syndetic $B \subseteq \N$, the set $A \cap B$ is dynamically (central) syndetic.  We will have immediate need for a more basic version of this result for cofinite sets in \cref{lemma_ds_dcs_modifyable_on_a_finite_set}.\\
\end{enumerate}

There are other directions, however, in which this definitional robustness does not extend.
\begin{enumerate}
\setcounter{enumi}{5}
    \item Minimality is a key feature of the definition.  For all $A \subseteq \N$, there exists a system $(X,T)$, a point $x \in X$, and a nonempty, open set $U \subseteq X$ containing $x$ such that $A = R(x,U)$.  Indeed, note that considering $A$ as a subset of $\N_0$, we have that $A = R(1_{A},[1]_0)$ in the full shift $(\{0,1\}^{\N_0},\sigma)$.  Thus, dropping the minimality assumption trivializes the definition of dynamical syndeticity.

    \item \label{item_generalization_to_z} That we are considering actions of the semigroup $(\N,+)$ instead of the group $(\Z,+)$ matters.  For example, if we define dynamically central syndetic subsets of $\Z$ analogously to those in $\N$, we see that all dynamically central syndetic subsets of $\Z$ contain 0.  In this case, if $x \in \overline{U} \setminus U$, then the set $R(x,U) \subseteq \Z$ is, in contrast to \cref{thm:simple_equivalent_dcS}, not dynamically central syndetic because it does not contain 0.  Other results from this paper for $(\N,+)$-systems fail to generalize under this definition, too, most notably the fact that thickly syndetic ($\PS^*$) sets (which do not necessarily contain 0) are dynamically central syndetic. Along with other evidence, this suggests that allowing $x \in \overline{U}$ in the definition of the family of dynamically central syndetic sets (as in \cref{thm:simple_equivalent_dcS} \eqref{item:very_strongly_central}) may yield the ``correct'' notion for more general (semi)group actions.
\end{enumerate}

\subsection{Definitional robustness}
\label{sec_dsyndetic_robust}

Here we show that invertibility and changes on finite sets do not affect the definition of the family of dynamically syndetic sets.

\begin{lemma}
\label{lemma_may_assume_invertible}
    A set $A \subseteq \N$ is dynamically (central) syndetic if and only if there exists an invertible system $(X,T)$, a point $x \in X$, and a nonempty, open set $U \subseteq X$ (containing $x$) such that $R(x,U) \subseteq A$.
\end{lemma}

\begin{proof}
    The ``if'' statement is immediate.  To see the ``only if'' statement, suppose $A \subseteq \N$ is dynamically syndetic.  There exists a system $(X,T)$, a point $x \in X$, and a nonempty, open set $U \subseteq X$ such that $R(x,U) \subseteq A$.  Let $\pi: (W,T) \to (X,T)$ be the natural extension (see \cite[Definition 6.8.10]{downarowiczbook}) of $(X,T)$, and let $w \in W$ be such that $\pi w = x$. Since $(X,T)$ is minimal, so is $(W,T)$ (see, for example, \cite[Lemma 2.9]{glasscock_koutsogiannis_richter_2019}). The system $(W,T)$ is invertible, and $R(w,\pi^{-1} U) = R(x,U)$.  Thus, the set $A$ contains $R(w,\pi^{-1} U)$ from the invertible system $(W,T)$.  Note that if $A$ is dynamically central syndetic and $x \in U$, then $w \in \pi^{-1} U$, whereby the set $R(x,U) = R(w,\pi^{-1} U)$ is dynamically central syndetic from the invertible system $(W,T)$.
\end{proof}

\begin{lemma}
    \label{lemma_ds_dcs_modifyable_on_a_finite_set}
    Let $A \subseteq \N$ be dynamically (central) syndetic.  If $B \subseteq \N$ is cofinite, then $A \cap B$ is dynamically (central) syndetic.
\end{lemma}

\begin{proof}
    By the definition of dynamically (central) syndetic, it suffices to show that for all minimal systems $(X,T)$, all points $x \in X$, all nonempty, open sets $U \subseteq X$ (containing $x$), and all $N \in \N$, the set $R(x,U) \setminus \{1, \ldots, N\}$ is dynamically (central) syndetic.  So, let $(X,T)$, $x \in X$, $U \subseteq X$ (contain $x$), and $N \in \N$ be as described.

    The set $V \defeq U \setminus \{Tx, T^2x, \ldots, T^N x\}$ is open and $R(x,U) \setminus \{1, \ldots, N\} \supseteq R(x, V)$.  If $V$ is nonempty, then we see that the set $R(x,U) \setminus \{1, \ldots, N\}$ is dynamically syndetic, as desired. Note that if $x \in \{Tx, T^2x, \ldots, T^N x\}$, then the system is periodic, $X = \{Tx, T^2x, \ldots, T^N x\}$, and $V$ is empty.  Therefore, if $V$ is nonempty and $x \in U$, then $x \in V$, and the set $R(x,U) \setminus \{1, \ldots, N\}$ is dynamically central syndetic, as desired.

    If, on the other hand, the set $V$ is empty, then $U \subseteq \{Tx, T^2x, \ldots, T^N x\}$.  Since $U$ is open, we see that some $T^i x$ is an isolated point of $X$.  By \cref{lemma_isolated_points_and_minimal_systems}, the system $(X,T)$ is finite and periodic.  In this case, the set $R(x,U) \setminus \{1, \ldots, N\}$ contains an infinite arithmetic progression (of the form $k \N$), which is dynamically (central) syndetic, as desired.
\end{proof}

\subsection{Translates and dilates}

In this section, we describe what happens to dynamically (central) syndetic sets under translation and dilation.  Recall that $A + n$ and $A-n$ are computed as subsets of $\N$, as described in \cref{sec_prelims}.

\begin{lemma}
\label{lemma_translates_of_dsyndetic_sets}
    Let $A \subseteq \N$ be dynamically syndetic.
    \begin{enumerate}
        \item
        \label{item_all_translates_are_ds}
        For all $n \in \N$, the set $A-n$ is dynamically syndetic.
        
        \item
        \label{item_some_translates_are_dcs}
        The set
        \begin{align}
            \label{eqn_translates_of_ds_set}
            \big\{ n \in \N \ \big| \ A-n \text{ is dynamically central syndetic} \big\}
        \end{align}
        is dynamically syndetic.
        
        \item
        \label{item_dcs_translates_of_dcs_are_dcs}
        If $A$ is dynamically central syndetic, then the set in \eqref{eqn_translates_of_ds_set} is dynamically central syndetic.
    \end{enumerate}
    Moreover, \eqref{item_all_translates_are_ds}, \eqref{item_some_translates_are_dcs}, and \eqref{item_dcs_translates_of_dcs_are_dcs} hold with $A-n$ replaced by $A+n$. Finally,
    \begin{align}
        \label{eqn_relationship_between_ds_and_dcs_classes}
        \dsyndetic = \bigcup_{n \in \N} \big(\dcsyndetic - n \big) = \bigcup_{n \in \N} \big(\dcsyndetic + n \big).
    \end{align}
\end{lemma}

\begin{proof}
    Since $A$ is dynamically syndetic, there exists a minimal system $(X,T)$, a point $x \in X$, and a nonempty, open set $U \subseteq X$ such that $R(x,U) \subseteq A$. According to \cref{lemma_may_assume_invertible}, we can assume the system $(X, T)$ is invertible.
    
    For $n \in \N$, we have that
    \begin{align}
        \label{eqn_translate_of_rxu_in_aminusn}
        R(T^n x, U) = R(x,U) - n \subseteq A - n.
    \end{align}
    Thus, the set $A-n$ is dynamically syndetic, demonstrating \eqref{item_all_translates_are_ds}.  It also follows from \eqref{eqn_translate_of_rxu_in_aminusn} that $R(x,U)$ is a subset of the set in \eqref{eqn_translates_of_ds_set}.  Indeed, if $n \in R(x,U)$, then $T^n x \in U$, so the set $A-n$ is dynamically central syndetic.  The set $R(x,U)$ is dynamically syndetic, whereby the set in \eqref{eqn_translates_of_ds_set} is dynamically syndetic, demonstrating \eqref{item_some_translates_are_dcs}.  All of this holds in the case that $A$ is dynamically central syndetic, but we have additionally that $R(x,U)$ is dynamically central syndetic, demonstrating \eqref{item_dcs_translates_of_dcs_are_dcs}.

    Now we will show that \eqref{item_all_translates_are_ds}, \eqref{item_some_translates_are_dcs}, and \eqref{item_dcs_translates_of_dcs_are_dcs} hold with $A-n$ replaced by $A+n$.
    
    Let $n \in \N$.  Recall that $T: X \to X$ is a homeomorphism. We claim that $R(x,U) + n = R(T^{-n}x, U) \cap \{n+1,n+2,\ldots\}$.  Indeed, note that $m \in R(x,U) + n$ if and only if $m - n \in R(x,U)$ if and only if $T^{m-n} x \in U$ and $m \geq n+1$ if and only if $T^m T^{-n}x \in U$ and $m \geq n+1$ if and only if $m \in R(T^{-n}x,U) \cap \{n+1, n+2, \ldots\}$.  By \cref{lemma_ds_dcs_modifyable_on_a_finite_set}, the set $R(T^{-n}x,U) \cap \{n+1,n+2,\ldots\}$ is dynamically syndetic.  Therefore, the set $A + n$ is dynamically syndetic, demonstrating \eqref{item_all_translates_are_ds}.
    
    Since $(X,T)$ is minimal and invertible, the system $(X,T^{-1})$ is minimal (see, for example, \cite[Lemma 2.7]{glasscock_koutsogiannis_richter_2019}). We claim that the set $R_{T^{-1}}(x,U)$ is contained in the set in \eqref{eqn_translates_of_ds_set}.  This will finish the proof of \eqref{item_some_translates_are_dcs} and \eqref{item_dcs_translates_of_dcs_are_dcs} since $R_{T^{-1}}(x,U)$ is dynamically syndetic and dynamically central syndetic when $x \in U$.  Suppose $n \in R_{T^{-1}}(x,U)$ so that $T^{-n} x \in U$.  From the previous paragraph, we see that $R(T^{-n}x, U) \cap \{n+1, n+2, \ldots\} = R(x,U) + n \subseteq A +n$.  By \cref{lemma_ds_dcs_modifyable_on_a_finite_set}, we have that the set $A + n$ is dynamically central syndetic, whereby $n$ belongs to the set in \eqref{eqn_translates_of_ds_set}, as desired.

    Finally, that \eqref{eqn_relationship_between_ds_and_dcs_classes} holds follows immediately from \eqref{item_all_translates_are_ds} and \eqref{item_some_translates_are_dcs} for $A-n$ and $A+n$.
\end{proof}

It is interesting to formulate the conclusions of \cref{lemma_translates_of_dsyndetic_sets} in terms of the family algebra developed in \cref{sec_combinatorics}. Thus,
\begin{enumerate}
    \item the family $\dsyndetic$ is translation invariant, that is, $\dsyndetic \subseteq \{\N\} + \dsyndetic$;
    \item $\dsyndetic \subseteq \dsyndetic + \dcsyndetic$; and
    \item the family $\dcsyndetic$ is idempotent, that is, $\dcsyndetic \subseteq \dcsyndetic + \dcsyndetic$.\\
\end{enumerate}

For the following lemma, recall the dilation notation set out at the beginning of \cref{sec_prelims}.

\begin{lemma}
\label{lemma_dilates_of_dsyndetic_sets}
    Let $A \subseteq \N$ and $k \in \N$. If $A$ is dynamically central syndetic, then so are the sets $kA$ and $A/k$.  If $A$ is dynamically syndetic, then so is the set $kA$.
\end{lemma}

\begin{proof}
    Let $A \in \dcsyndetic$. Let $(X, T)$ be a minimal system, $x \in X$, and $U \subseteq X$ be an open neighborhood of $x$ such that $A \supseteq R(x, U)$. We will show that $k A \in \dcsyndetic$ by considering a suspension system with base $(X, T)$. More precisely, let $Y = X \times \{0, 1, \ldots, k - 1\}$, and define $S: Y \to Y$ by 
    \[
        S(x', \eta) = \begin{cases} (x', \eta + 1) \text{ if } \eta \leq k - 2, \\
        (Tx', 0) \text{ if } \eta = k - 1.
        \end{cases}
    \]
    It is easy to see that the orbit of every point in $Y$ is dense, whereby $(Y, S)$ is minimal.  Also, we see that $k R_T(x, U) = R_S ((x, 0), U \times \{0\}) \in \dcsyndetic$, as desired.

    To see that $A/k \in \dcsyndetic$, we appeal to a well-known fact (cf. \cite[Cor. 3.2]{Ye-D_function}): if a point in a system is uniformly recurrent under the transformation, then it is uniformly recurrent under every power of the transformation.  Since $(X,T)$ is minimal, the point $x$ is uniformly recurrent under $T$, and hence uniformly recurrent under $T^k$.  Therefore, the system $(\overline{T^{k \N_0} x}, T^k)$ is minimal, contains $x$, and $A/k \supseteq R_{T^k}(x, U \cap \overline{T^{k \N_0} x}) \in \dcsyndetic$, as desired.
    
    Suppose $A \in \dsyndetic$. By \cref{lemma_translates_of_dsyndetic_sets} \eqref{item_some_translates_are_dcs}, there exists $n \in \N$ such that $A - n \in \dcsyndetic$. By the first paragraph of this proof, we have $k(A - n) \in \dcsyndetic$. Therefore, by \cref{lemma_translates_of_dsyndetic_sets} \eqref{item_all_translates_are_ds}, we see $k A \supseteq k(A - n) + kn \in \dsyndetic$, whereby $kA \in \dsyndetic$, as desired.
\end{proof}

\subsection{Symbolic characterizations}

We show in this section that the family of dynamically (central) syndetic sets is not affected by requiring return time sets to come from zero-dimensional systems.  This is accomplished by considering the indicator functions $1_A$ of sets $A \subseteq \N$ in the symbolic space $\{0,1\}^{\N}$.

As a warm-up to the challenges that arise in \cref{thm:simple_equivalent_dcS}, it is worth noting that when $(X,T)$ is a minimal system, $x \in X$, and $U \subseteq X$ is nonempty and open, the point $1_{R(x,U)} \in \{0,1\}^{\N}$ is not always uniformly recurrent under the shift map $\sigma$, as the following example demonstrates.

\begin{example}
\label{ex:never_repeat}
    Let $\alpha \in (0, 1)$ be irrational.  Let $X \defeq \R / \Z$ and $T: X \to X$ be defined by $T(x) = x + \alpha$.  Let $U \defeq X \setminus \{\alpha\}$ and $x \defeq 0$. We see that
    \[
        1_{R(0,U)} = (1_{U}(T^n x))_{n \in \N} = (1_U( \alpha), 1_U(2 \alpha), 1_U(3 \alpha), \ldots) = (0, 1, 1, \ldots).
    \]
    Since $\alpha$ is irrational, for all $n \geq 2$, $T^n \alpha \neq \alpha$. Therefore, the function $1_{R(0,U)}$ takes the value $0$ only at $1$.  The point $1_{R(0,U)} \in [0]_1$ is not uniformly recurrent because it never returns to the neighborhood $[0]_1$.
\end{example}

In the following theorem, point \eqref{item:very_strongly_central} says that the set $A$ is \emph{very strongly central}, according to the terminology in \cite[Def. 2.10]{bergelson_hindman_strauss_2012}.

\begin{theorem}
\label{thm:simple_equivalent_dcS}
    Let $A \subseteq \N$. The following statements are equivalent.
    \begin{enumerate}
        \item \label{item:dcsyndetic} The set $A$ is dynamically central syndetic.  

        \item \label{item:very_strongly_central}
        There exists a minimal system $(X,T)$, a nonempty, open set $U \subseteq X$, and a point $x \in \overline{U}$ such that $R(x,U) \subseteq A$.

        \item \label{item:1_A_0_uniformly_recurrent} There exists a subset $B \subseteq A$ such that the point $1_{B \cup \{0\}}$ is uniformly recurrent in the full shift $(\{0, 1\}^{\N_0}, \sigma)$. In particular, the system $(X \defeq \overline{\sigma^{\N_0} 1_{B \cup \{0\}}}, \sigma)$ is minimal, the set $X \cap [1]_0$ is a clopen neighborhood of $1_{B \cup \{0\}}$, and $B = R(1_{B \cup \{0\}}, X \cap [1]_0) \subseteq A$.
    \end{enumerate}
\end{theorem}

\begin{proof}
    (\ref{item:dcsyndetic}) $\Rightarrow$ (\ref{item:very_strongly_central}) \ This follows by definition.

    (\ref{item:very_strongly_central}) $\Rightarrow$ (\ref{item:1_A_0_uniformly_recurrent}) \ Assume that $A \supseteq R(x, U)$, where $(X, T)$ is a minimal system, $U \subseteq X$ is nonempty and open, and $x \in \overline{U}$.
    
    If $x$ is a periodic point of $(X, T)$, let $n_0 \in \N$ be least such that $T^{n_0} x = x$. Since $(X, T)$ is minimal, $X = \{x, Tx, \ldots, T^{n_0 - 1} x\}$. As $X$ is finite and $x \in \overline{U}$, it must be that $x \in U$. Letting $B = n_0 \N$ be the set of positive multiples of $n_0$, we have that $A \supseteq R(x, U) \supseteq B$ and that $1_{B \cup \{0\}}$ is a periodic, hence uniformly recurrent, point in $\{0, 1\}^{\N_0}$.

    Suppose that $x$ is not a periodic point of $(X,T)$. We claim that there exists a nonempty, open set $V \subseteq U$ such that $x \in \overline{V}$ and the boundary $\partial V$ is disjoint from $T^{\N} x \defeq \{T^n x \ | \ n \in \N\}$.

    If $x \in U$, choose $V = B_\delta(x)$ for some small $\delta > 0$. Note that the boundary of $B_\delta(x)$ is the set $\{z \in X \ | \ d_X(x, z) = \delta\}$, which may be empty. Since there are uncountably many choices for $\delta$, there exists a choice so that the boundary is disjoint from $T^{\N} x$.

    Suppose that $x \in \partial U = \overline{U} \setminus U$.  There exists a sequence of pairwise distinct points $(x_k)_{k \in \N}$ in $U \setminus \{x\}$ such that $\lim_{k \to \infty }x_k = x$.
    Inductively, choose a sequence of disjoint balls $B_{\delta_k}(x_k) \subseteq U$ with $\lim_{k \to \infty} \delta_k = 0$ such that the boundary of each $B_{\delta_k}(x_k)$ has empty intersection with $T^{\N} x$.
    Let $V = \bigcup_{k = 1}^{\infty} B_{\delta_k}(x_k)$. It is clear that $V$ is an open subset of $U$. Since $\lim_{k \to \infty }x_k = x$, we have that $\partial V = \{x\} \cup \bigcup_{k=1}^{\infty} \partial V_k$, and so the boundary of $V$ is disjoint from $T^{\N} x$. Our claim is proved.
    
    Let $B = R(x, V)$. Since $V$ is nonempty and $V \subseteq U$, we have that $B$ is nonempty and that $B \subseteq A$. It remains to show that $1_{B \cup \{0\}}$ is a uniformly recurrent point in $(\{0, 1\}^{\N_0}, \sigma)$.  By the topology on the space $\{0, 1\}^{\N_0}$, we must show that for all $L \in \N$, there exists a syndetic set of positive integers $m$ such that
    \begin{align}
    \label{eq:syndetic_return_condition}
        \text{for all $i \in \{0, \ldots, L\}$}, \ 1_{B \cup \{0\}}(m+i) = 1_{B \cup \{0\}}(i).
    \end{align}
    Let $L \in \N$.  As $\partial V \cap T^{\N} x = \emptyset$, for every $n \in \N$, the point $T^n x$ is either in $V$ or in $X \setminus \overline{V}$.
    For $n \in \{1, \ldots, L\}$, define
    \[
        W_n = \begin{cases} V & \text{ if } T^n x \in V \\ X \setminus \overline{V} & \text{ if } T^n x \in X \setminus \overline{V}. \end{cases}
    \]
    It follows that $x \in \bigcap_{n = 1}^L T^{-n} W_n$. Since each $W_n$ is open, the set $\bigcap_{n = 1}^L T^{-n} W_n$ is an open neighborhood of $x$. Because $x \in \overline{V}$, the set $V \cap \bigcap_{n = 1}^L T^{-n} W_n$ is nonempty and open. Since $(X, T)$ is minimal, the set
    \[R \left( x, V \cap \bigcap_{n = 1}^L T^{-n} W_n \right)\]
    is syndetic, and all integers $m$ in this set satisfy the condition in \eqref{eq:syndetic_return_condition}. This demonstrates that $1_{B \cup \{0\}}$ is uniformly recurrent under the shift, as desired.
    
    (\ref{item:1_A_0_uniformly_recurrent}) $\Rightarrow$ (\ref{item:dcsyndetic}) \ Suppose $A$ contains a subset $B$ for which $1_{B \cup \{0\}}$ is uniformly recurrent in $(\{0, 1\}^{\N_0}, \sigma)$. 
    The system $(X \defeq \overline{\sigma^{\N_0} 1_{B \cup \{0\}}}, \sigma)$ is minimal, and the cylinder set $X \cap [1]_0$ is open and contains $1_{B \cup \{0\}}$.  Since $A \supseteq B = R(1_{B \cup \{0\}}, X \cap [1]_0)$, we have that $A$ is dynamically central syndetic, as desired.
\end{proof}

A version of the following theorem is given in \cite[Prop. 2.3]{huang_ye_2005}.  The derivation from \cref{thm:simple_equivalent_dcS} is short, so we give it here.

\begin{theorem}
\label{thm:simple_equivalent_dS}
    Let $A \subseteq \N$.  The following statements are equivalent.
    \begin{enumerate}
        \item \label{item:dsyndetic_one} The set $A$ is dynamically syndetic.
        
        \item \label{item:1_A_uniform_recurrent_one} There exists a nonempty subset $B \subseteq A$ for which the point $1_B$ is uniformly recurrent in the full shift $(\{0, 1\}^{\N},\sigma)$.  Moreover, when $B$ is considered as a subset of $\N_0$, the point $1_B$ is uniformly recurrent in the full shift $(\{0, 1\}^{\N_0},\sigma)$, the system $(X \defeq \overline{\sigma^{\N_0} 1_B}, \sigma)$ is minimal, the set $X \cap [1]_0$ is nonempty and clopen, and $B = R(1_{B}, X \cap [1]_0) \subseteq A$.
    \end{enumerate}
\end{theorem}

\begin{proof}
    (\ref{item:dsyndetic_one}) $\Rightarrow$ (\ref{item:1_A_uniform_recurrent_one}) \ Suppose $A$ is dynamically syndetic. By \cref{lemma_translates_of_dsyndetic_sets}, there exists $k \in \N$ such that $A + k$ is dynamically central syndetic. By \cref{thm:simple_equivalent_dcS}, there exists a set $\tilde B \subseteq A + k$ such that $1_{\tilde B \cup \{0\}}$ is uniformly recurrent in $(\{0, 1\}^{\N_0},\sigma)$. Note that $B \defeq \tilde B - k \subseteq A$.  Since $1_{\tilde B \cup \{0\}}$ is uniformly recurrent in $(\{0, 1\}^{\N_0},\sigma)$, we have that $1_{\tilde B}$ is uniformly recurrent in $(\{0, 1\}^{\N},\sigma)$.  Applying $\sigma^k$, we have that $\sigma^k 1_{\tilde B} = 1_{B}$ is uniformly recurrent, as was claimed.
    
    Since $1_{\tilde B \cup \{0\}}$ is uniformly recurrent in $(\{0, 1\}^{\N_0},\sigma)$, we have that $\sigma^k 1_{\tilde B \cup \{0\}} = 1_{B}$, where $B$ is interpreted as a subset of $\N_0$, is uniformly recurrent in $(\{0, 1\}^{\N_0},\sigma)$.  Thus, the system $(X,\sigma)$ is minimal.  Since the set $\tilde B$ is infinite (in fact, it is syndetic), the set $B$ is nonempty, and hence, by the definition of $X$, the set $X \cap [1]_0$ is nonempty.  Finally, that $B = R(1_{B}, X \cap [1]_0)$ is straightforward to check.

    (\ref{item:1_A_uniform_recurrent_one}) $\Rightarrow$ (\ref{item:dsyndetic_one}) \ 
    Suppose $A$ contains a nonempty subset $B$ for which $1_{B}$ is uniformly recurrent in $(\{0, 1\}^{\N},\sigma)$. Let $k = \min B$. Since $1_{B}$ is uniformly recurrent, so is $\sigma^k 1_B = 1_{B-k}$.  Since $k \in B$, we have that $1_{(B-k) \cup \{0\}}$ is uniformly recurrent in $\{0, 1\}^{\N_0}$. By \cref{thm:simple_equivalent_dcS}, the set $B - k \in \dcsyndetic$. Since $A - k \supseteq B - k$, the set $A - k$ is dynamically central syndetic.  By \cref{lemma_translates_of_dsyndetic_sets}, the set $A \supseteq (A-k) + k$ is dynamically syndetic, as desired.
\end{proof}

\section{Central and dynamically central syndetic sets}
\label{sec_central_and_dy_central_sets}

In this section, we prove \cref{thm_dcs_iff_cs}, that dynamically central syndetic sets are characterized by membership in syndetic, idempotent filters.  This is part of the statement of \cref{intro_mainthm_full_chars_of_dcsyndetic}. The following definition will help streamline the discussion and results.

\begin{definition}
\label{def_central_syndetic}
    A set $A \subseteq \N$ is a \emph{central syndetic} if it is a member of a syndetic, idempotent filter on $\N$.
\end{definition}

Let us inject a historical remark to motivate the nomenclature. Central sets -- described first by Furstenberg  \cite[Def. 8.3]{furstenberg_book_1981} in a dynamical context -- have since been characterized as members of minimal, idempotent ultrafilters \cite{bergelson_hindman_1990,glasner_1980,shi_yang_1996}. 
To discuss the dynamical and combinatorial families separately, it has become customary (see, eg., \cite[Sec. 19.3]{hindman_strauss_book_2012}) to call Furstenberg's sets ``dynamically central'' and members of minimal, idempotent ultrafilters ``central.''  We take a similar approach here: ``dynamically central syndetic'' sets are defined dynamically as above, while ``central syndetic'' sets are defined combinatorially as members of syndetic, idempotent filters.  We will demonstrate in \cref{thm_dcs_iff_cs} that these families are the same.

To further motivate the nomenclature, it would be natural to show that central syndetic sets are, in fact, central. Here is one way to see it: central syndetic sets are dynamically central syndetic (\cref{intro_mainthm_full_chars_of_dcsyndetic}); dynamically central syndetic sets are dynamically central (from the definitions); and dynamically central sets are central (as mentioned in the previous paragraph). We give a direct argument in Section 6.1 of our companion paper \cite{glasscock_le_2025}.

\begin{theorem}
\label{thm_dcs_iff_cs}
    A set $A \subseteq \N$ is dynamically central syndetic if and only if it is central syndetic, that is, a member of a syndetic, idempotent filter on $\N$.
\end{theorem}

\begin{proof}[Proof of the ``only if'' statement in \cref{thm_dcs_iff_cs}]
    By assumption, there exists a minimal system $(X,T)$, a point $x \in X$, and an open set $U \subseteq X$ containing $x$ such that $R(x,U) \subseteq A$. Define
    \[
        \filter \defeq \upclose \big\{ R(x,V) \ \big| \ V \subseteq X \text{ is an open set containing $x$} \big\}.
    \]
    We claim that $\filter$ is a syndetic, idempotent filter containing $A$.

    That $\filter$ is a filter, that $A \in \filter$, and that $\filter$ is syndetic are all immediate.  To see that $\filter$ is idempotent, it suffices by \cref{rmk_showing_idempotency} to show that for all open $V \subseteq X$ with $x \in V$, we have that $R(x,V) \subseteq R(x,V) - \filter$. This is the containment shown in \eqref{eqn_inequality_to_show_idempotency} in \cref{sec_intro_algebra_of_filters}.
\end{proof}

The rest of this section is devoted to a proof of the ``if'' direction of \cref{thm_dcs_iff_cs}, namely that central syndetic sets are dynamically central syndetic.

\subsection{A combinatorial characterization of central syndetic sets}
\label{sec_combo_condition_for_central_syndetic}

In this subsection, we give a combinatorial characterization of central syndetic sets.  Consider the following statement for a set $B \subseteq \N$:
\begin{align}
    \label{eqn_def_of_cssd}
    \text{for all finite $F \subseteq B$, \ the set } B \cap \bigcap_{f \in F} \big( B - f \big) \text{ is syndetic}.
\end{align}
The following theorem shows that the condition in \eqref{eqn_def_of_cssd} captures membership in a syndetic, idempotent filter.

\begin{theorem}
\label{thm_conditions_for_central_syndetic}
    Let $A \subseteq \N$.
    \begin{enumerate}
        \item \label{item_central_syndetic_implies_comb_condition}
        If $A$ is central syndetic and $\filter$ is a syndetic, idempotent filter to which $A$ belongs, then the set $A \cap (A-\filter)$ satisfies the condition in \eqref{eqn_def_of_cssd}.

        \item \label{item_comb_condition_implies_central_syndetic}
        If $A$ satisfies the condition in \eqref{eqn_def_of_cssd}, then the family
        \[\filter \defeq \bigupclose \bigg\{ A \cap \bigcap_{f \in F} (A - f) \ \bigg | \ F \subseteq A \text{ is finite} \bigg\}\]
        is a syndetic, idempotent filter that contains $A$, whereby $A$ is central syndetic.
    \end{enumerate}
    In particular, the set $A$ is central syndetic if and only if some subset of it satisfies the condition in \eqref{eqn_def_of_cssd}.
\end{theorem}

\begin{proof}
    \eqref{item_central_syndetic_implies_comb_condition} \ Suppose $A$ belongs to a syndetic, idempotent filter $\filter$.  Define $A_1 \defeq A \cap (A-\filter)$.  We claim that this set satisfies the condition in \eqref{eqn_def_of_cssd}.

    First we will show that $A_1 \in \filter$ and $A_1 \subseteq A_1 - \filter$.  That $A_1 \in \filter$ follows because $A \in \filter$ and $A-\filter \in \filter$ (since $\filter$ is idempotent).  Because $\filter \subseteq \filter + \filter$, using the algebra in \cref{lemma_filter_algebra}, we see
    \begin{align*}
        A_1 = A \cap (A-\filter) \subseteq A - \filter &\subseteq (A - \filter) \cap \big(A - (\filter + \filter) \big) \\
        &= (A - \filter) \cap \big((A - \filter) - \filter \big) \\
        &= \big( A \cap (A - \filter) \big) - \filter \\
        &= A_1 - \filter.
    \end{align*}

    To see that the condition in \eqref{eqn_def_of_cssd} is satisfied, let $F \subseteq A_1$ be finite.  Since $F \subseteq A_1 - \filter$, the set $\bigcap_{f \in F} (A_1 - f) \in \filter$.  Since $A_1 \in \filter$, we have that $A_1 \cap \bigcap_{f \in F} (A_1 - f) \in \filter$, whereby it is syndetic, as desired.

    \eqref{item_comb_condition_implies_central_syndetic} \ That $\filter$ is a filter is immediate from its definition.  That $A \in \filter$ follows immediately from the fact that $\filter$ is upward closed.

    To see that $\filter$ is syndetic, it suffices to show that for all $F \subseteq A$ finite, the set $B \defeq A \cap \bigcap_{f \in F} (A - f)$ is syndetic.  But this is immediate from \eqref{eqn_def_of_cssd}.

    To see that $\filter$ is idempotent, we must show that $\filter \subseteq \filter + \filter$.  Since $\filter + \filter$ is a filter, it suffices to show that for all $a \in A \cup\{0\}$, the set $A - a \in \filter + \filter$.
    
    Let $a \in A \cup\{0\}$.  We will show that $(A - a) - \filter \in \filter$.  Since $A - a \in \filter$, it suffices to show that $A - a \subseteq (A - a) - \filter$.  This holds since for any $n \in A - a$, we have that $(A - a) - n = A - (n + a)$, which is a member of $\filter$ by the definition of $\filter$ since $n+a \in A$.
\end{proof}

\begin{remark}
    It is a short exercise to show that the condition in \eqref{eqn_def_of_cssd} is equivalent to the following: \emph{there exists a chain of syndetic sets $A \supseteq A_1 \supseteq A_2 \supseteq A_3 \supseteq \cdots $ with the property that for all $n \in \N$ and all $a \in A_n$, there exists $m \in \N$ such that $A_n - a \supseteq A_m$.}  The combinatorial characterization of central sets in \cite[Thm. 14.25]{hindman_strauss_book_2012} says the same thing with with ``syndetic'' replaced by ``collectionwise piecewise syndetic.''
\end{remark}

\begin{lemma}
\label{lemma_lots_of_self_shifts_of_cssd_set}
    If $A \subseteq \N$ is contained in a syndetic, idempotent filter, then for all thick sets $T \subseteq \N$, the set $A \cap T$ is IP.
\end{lemma}

\begin{proof}
    Because the desired conclusion remains unchanged under taking supersets, we may assume without loss of generality that the set $A$ satisfies the condition in \eqref{eqn_def_of_cssd}.
    
    Let $T \subseteq \N$ be thick.  Since $A$ is syndetic and $T$ is thick, there exists $a_1 \in A \cap T$.  By the condition in \eqref{eqn_def_of_cssd}, the set $A \cap (A-a_1)$ is syndetic.  Since $T$ is thick, so is the set $T \cap (T-a_1)$, so the set $A \cap (A-a_1) \cap T \cap (T-a_1)$ is nonempty.  Take $a_2 \in A \cap (A-a_1) \cap T \cap (T-a_1)$.  Since $\{a_1, a_2, a_1+a_2\} \subseteq A$, again by \eqref{eqn_def_of_cssd} and the thickness of $T$, we have that
    \[A \cap (A-a_1) \cap (A-a_2) \cap (A-(a_1+a_2)) \cap T \cap (T-a_1) \cap (T-a_2) \cap (T-(a_1+a_2)) \neq \emptyset.\]
    Keep iterating to find that $\text{FS}(a_1, a_2, \ldots) \subseteq A \cap T$, as desired.
\end{proof}

We finish this subsection by showing that sets $B \subseteq \N$ satisfying the condition in \eqref{eqn_def_of_cssd} satisfy the ostensibly stronger condition:
\begin{align}
    \label{eqn_stronger_than_cssd}
    \text{for all finite $F \subseteq B$ and all $m \in \N$, the set } B \cap \bigcap_{f \in F} \big( B - f \big) \cap m \N \text{ is syndetic}.
\end{align}
This upgrade will help us with the proof of \cref{thm_dcs_iff_cs} in the next section.

\begin{lemma}
\label{lemma_upgraded_cssd_property}
    A set $A \subseteq \N$ satisfies the condition in \eqref{eqn_def_of_cssd} if and only if it satisfies the condition in \eqref{eqn_stronger_than_cssd}.
\end{lemma}

\begin{proof}
    If $A$ satisfies the condition in \eqref{eqn_stronger_than_cssd}, then it clearly satisfies the one in \eqref{eqn_def_of_cssd}.

    Conversely, suppose that $A \subseteq \N$ satisfies the condition in \eqref{eqn_def_of_cssd}. To see that it satisfies \eqref{eqn_stronger_than_cssd}, let $F \subseteq A$ be finite and $m \in \N$. To see that the set $A \cap \bigcap_{f \in F} (A-f) \cap m\N$ is syndetic, we will show that it has nonempty intersection with all thick sets.
    
    Let $H \subseteq \N$ be thick. By \cref{thm_conditions_for_central_syndetic}, we see that the set $A \cap \bigcap_{f \in F} (A-f)$ belongs to a syndetic, idempotent filter, and hence is central syndetic.  It follows from \cref{lemma_lots_of_self_shifts_of_cssd_set} that the set $A \cap \bigcap_{f \in F} (A-f) \cap H$ is IP.  The set $m \N$ is IP$^*$, and so the set $A \cap \bigcap_{f \in F} (A-f) \cap m \N \cap H$ is nonempty, as was to be shown.
\end{proof}

\subsection{Central syndetic sets are dynamically central syndetic}
\label{sec_main_idea_of_shift_punch}

Here is the main idea behind the proof of the ``if'' statement in \cref{thm_dcs_iff_cs}.  By \cref{thm_conditions_for_central_syndetic}, we may assume that the set $A$ satisfies the condition in \eqref{eqn_def_of_cssd}.

Our aim is to show that the set $A$ contains a set of the form $R(x,U)$, where the point $x$ belongs to the open set $U$ in a minimal system.  In view of \cref{thm:simple_equivalent_dcS}, it suffices to find a subset $B \subseteq A$ for which $1_{B \cup \{0\}} \in \{0, 1\}^{\N_0}$ is uniformly recurrent under the shift map $\sigma$.  Note that the condition in \eqref{eqn_def_of_cssd} is not sufficient for $1_{A \cup \{0\}}$ to be uniformly recurrent under $\sigma$.  (Indeed, the set $\N_0 \setminus \{n^2 \ | \ n \in \N\}$ satisfies the condition in \eqref{eqn_def_of_cssd} but has an indicator function that is not uniformly recurrent under $\sigma$.)  To find the set $B$, the idea is to modify the shift map $\sigma$ into a ``shift-punch'' map $\kappa$ under which the point $1_A$ becomes uniformly recurrent.

We append the integer $0$ to $A$ and consider it as a subset of $\N_0$. The ``punch map'' $\pi: \{0,1\}^{\N_0} \to \{0,1\}^{\N_0}$, which depends on the set $A$, is described roughly as follows.  At $1_C \in \{0,1\}^{\N_0}$, we find the maximal $N \in \N$ such that $A \cap [N] \subseteq C$.  Then, $\pi(1_C) = 1_D$, where $D \cap [N] = A \cap [N]$ and $D \cap (\N_0 \setminus [N]) = C \cap (\N_0 \setminus [N])$.  Roughly, the map $\pi$ ``punches out'' or ``deletes'' the elements of $C$ on the interval $[N]$ that are not in $A \cap [N]$.  The map $\pi$ continuously moves points of $\{0,1\}^{\N_0}$ closer to $1_A$.

The shift-punch map $\kappa$ is defined to be $\pi \circ \sigma$.  The dynamics of the shift-punch system $(\{0,1\}^{\N_0}, \kappa)$ is thus a shift followed by a punch.  The idea is to show that in this system, the point $1_A$ is uniformly recurrent.  This happens because the set $A$ satisfies the condition in \eqref{eqn_def_of_cssd} and thus syndetically often contains an initial interval of itself.  Keeping track of which elements of $A$ are deleted as the shift-punch map is repeatedly applied to $1_A$, we end up with a subset $B \subseteq A$. We will show that the point $1_B \in \{0,1\}^{\N_0}$ is uniformly recurrent under the usual shift map, as desired.

The outline just given is difficult to control in practice because the punches can seemingly overlap in complicated ways.  To overcome this, we find it useful to have more control over the length of intervals on which the punches occurs.  Thus, the actual shift-punch system is defined as a skew product over a 2-adic odometer.  The odometer functions only to control the length of the punch intervals, which are restricted to occur only on dyadic intervals.  Dyadic intervals are convenient because a nonempty intersection implies containment.  This control on the punch intervals allows us to realize the outline above.

\subsubsection{Notation and spaces}
\label{sec_notation_and_spaces}

In the rest of this section, we will alternate between working in $\N$ and $\N_0 \defeq \{0,1,2,\ldots\}$.  Recall that for $N \in \N_0$, we write $[N] \defeq \{0, \ldots, N-1\}$. When $A \subseteq \N_0$, we compute the translate $A-n$ as a subset of $\N_0$.

It will be convenient to append some elements to $\N_0$.  Thus, we define $\N_{0,\pm \infty} \defeq \N_0 \cup \{-\infty, \infty\}$.  Subspaces, such as $\N_{0,\infty} \defeq \N_0 \cup \{\infty\}$, are defined analogously.  We give $\N_{0,\pm \infty}$ the topology that makes the bijection $\N_{0,\pm \infty} \to \{2\} \cup \{1/n \ | \ n \in \N\} \cup \{0\}$, where $-\infty \mapsto 2$, $x \mapsto 1/(x+1)$, and $\infty \mapsto 0$ a homeomorphism.  Thus, every point of $\N_{0,\pm \infty}$, except for $\infty$, is isolated.  Subspaces of $\N_{0,\pm \infty}$ are given the subspace topology.

\subsubsection{Dyadic valuation and intervals}

Denote by $\nu_2: \Z \to \N_{0,\infty}$ the two-adic valuation on $\Z$, defined by $\nu_2(n) \defeq \max \{k \in \N_0 \ | \ 2^k \text{ divides } n \}$. A finite interval $\{a, a+1, \ldots, a+\ell - 1\}$ in $\N_0$ is a \emph{dyadic interval} if $\ell = 2^k$ for some $k \in \N_0$ and $\ell \mid a$.  Denote by $\cubes_{2^n}$ the set of dyadic intervals of cardinality $2^n$.  We will say that $\N_0$ is the dyadic interval of infinite cardinality.

The following lemma follows from standard facts about dyadic intervals.  We provide a proof for completeness and ease of reference for later.

\begin{lemma}
\label{lemma_basic_dyadic_cubes_facts}
Let $Q$ and $P$ be finite dyadic intervals in $\N_0$.
\begin{enumerate}
    \item \label{item_dyadic_prop_two}
    $Q \cap 2^{\nu_2(\min Q)}\N = \{\min Q\}$.
    \item \label{item_dyadic_prop_three}
    For all $\ell \in \N_0$, if $\big(Q \setminus \{\min Q\} \big) \cap 2^{\ell} \N \neq \emptyset$, then $|Q| \geq 2^{\ell + 1}$.
    \item \label{item_dyadic_prop_one}
    If $P \cap Q \neq \emptyset$, then either $P \subseteq Q$ or $Q \subseteq P$.
\end{enumerate}
\end{lemma}

\begin{proof}
    
    \eqref{item_dyadic_prop_two} \ There exists $n \in \N$ and $0 \leq k \leq \nu_2(n)$ such that $Q = n + [2^k]$.  Since $Q$ is an interval beginning at $n \in 2^{\nu_2(n)} \N$ with $|Q| = 2^k \leq 2^{\nu_2(n)}$, we have that $Q \cap 2^{\nu_2(n)}\N = \{n\}$.

    \eqref{item_dyadic_prop_three} \ Write $Q = n + [2^k]$ as in \eqref{item_dyadic_prop_two}. Let $\ell \in \N_0$, and suppose that $\big(Q \setminus \{n\} \big) \cap 2^{\ell} \N$ is nonempty.  It follows from the previous paragraph that $\ell < \nu_2(n)$.  Let $a \in \big(Q \setminus \{n\} \big) \cap 2^{\ell} \N$.  Since $|Q| = 2^k > a-n$ and $\nu_2(a-n) \geq \ell$, we see that $k \geq \ell + 1$.  Therefore, $|Q| \geq 2^{\ell + 1}$, as desired.

    \eqref{item_dyadic_prop_one} \ Let $P$ and $Q$ be dyadic intervals.  Note that if $\min P = \min Q$, then $P \subseteq Q$ or $Q \subseteq P$, and we are done.  Otherwise, we will show that if $\min P \in Q \setminus \{\min Q\}$, then $P \subseteq Q$.  It will follow that either $P$ and $Q$ are disjoint or that one is contained in the other.  
    
    Suppose for a contradiction that $\min P \in Q \setminus \{\min Q\}$ and $P \not\subseteq Q$.  Then $\max Q + 1 = \min Q + |Q| \in P$ and is divisible by $|Q|$.  It follows from \eqref{item_dyadic_prop_three} that $|P| \geq 2 |Q|$.  But since $P \not\subseteq Q$ and $|P|$ divides $\min P$, we see by \eqref{item_dyadic_prop_three} that $|Q| \geq 2 |P|$, a contradiction.
\end{proof}

We denote by $\cubes_{2^n}^*$ the collection of subsets of $\N$ that have nonempty intersection with all members of $\cubes_{2^n}$.  (This notation matches notation for the dual family defined in \cref{sec_dual_families}, but note that $\cubes_{2^n}$ is not a family since it is not upward closed.)

\begin{lemma}
\label{lemma_syndetic_characterization}
A set $A \subseteq \N_0$ is syndetic if and only if there exists $k \in \N$ such that $A \in \cubes_{2^k}^*$.
\end{lemma}

\begin{proof}
    This is left to the reader as an easy exercise.
\end{proof}

\subsubsection{The 2-adic odometer}
\label{sec_two_adic_odometer}

We denote by $(\odometer,\orotation)$ the 2-adic odometer, the inverse limit of the family of rotations $(\Z / 2^k \Z, \allowbreak +1)$, $k \in \N_0$, as topological dynamical systems.  We do not have need to specify a metric on $\odometer$ explicitly; it suffices for our purposes to know that two elements of $\odometer$ are near if and only if their projections to $\Z / 2^k \Z$ agree for some large value of $k$.

For $n \in \Z$, we denote by $n$ the element of $Z$ that projects, for all $k \in \N_0$, to $n \in \Z / 2^k \Z$. This association gives a dense embedding of $\Z$ into $\odometer$.  By a convenient abuse of notation, we write $\Z \subseteq \odometer$, identifying $\Z$ with this copy in $\odometer$.

The two-adic valuation $\nu_2: \Z \to \N_{0,\infty}$ is uniformly continuous in the subspace topology that $\Z$ inherits from the space $\odometer$.  Indeed, if the images of integers $n$ and $m$ are the same in $\Z / 2^k \Z$ for a large value of $k$, then $2^k$ divides $n-m$, so $\nu_2(n)$ is large if and only if $\nu_2(m)$ is large.  Therefore, the map $\nu_2: \Z \to \N_{0,\infty}$ extends uniquely to a continuous map $\nu_2: \odometer \to \N_{0,\infty}$.

\subsubsection{The punch maps}
\label{sec_punch_maps}

Let $X = \{0, 1\}^{\N_0}$, and fix $A \subseteq \N_0$.  We will define a family of continuous maps $\pi_n: X \to X$, $n \in \N_{0,\infty}$, depending on $A$ that drive the shift-punch dynamics described in the next section.

First, we define a map $\nu_A: X \to \N_{0,\pm\infty}$ in the following way.  For $\omega \in X$,
\begin{align}
    \label{eqn_def_of_M_set}
    \nu_A(\omega) \defeq \sup \big\{n \in \N_0 \ \big| \ A \cap [2^{n}] \subseteq \supp(\omega) \big\},
\end{align}
where the supremum of the empty set is $-\infty$.  Note that $\nu_A(\omega) = -\infty$ if and only if $0 \in A$ and $\omega(0) = 0$.  Also, note that $\nu_A(\omega) = \infty$ if and only if $A \subseteq \supp(\omega)$.

For $n \in \N_{0,\infty}$ and $\omega \in X$, we define the \emph{punch interval}
\[I_n(\omega) \defeq \big[2^{\min(\nu_A(\omega),n)} \big],\]
where $[2^{-\infty}]$ and $[2^{\infty}]$ are interpreted to be the empty set and $\N_0$, respectively. Then, we define the \emph{punch map} $\pi_n: X \to X$ by
\[\big(\pi_n (\omega) \big)(i) = \begin{cases} 1_A(i) & \text{ if $i \in I_n(\omega)$} \\ \omega(i) & \text{ if $i \in \N_0 \setminus I_n(\omega)$} \end{cases}.\]
Informally, if $A \cap [2^m] \subseteq \supp(\omega)$, where $m \in \{-\infty, 0, 1, \ldots, n\}$ is maximal as such, then $\pi_n(\omega)$ is defined so that $\supp(\pi_n(\omega)) \cap [2^m] = A \cap [2^m]$ and $\supp(\pi_n(\omega)) \cap (\N_0 \setminus [2^m]) = \supp(\omega) \cap (\N_0 \setminus [2^m])$. The word ``punch'' was chosen since the punch maps ``punch out'' ones by changing them to zeroes.

\begin{lemma}
\label{lemma_continuity_of_family_of_punches}
    Each of the maps $\pi_n: X \to X$, $n \in \N_{0,\infty}$, is continuous, and the map $\N_{0,\infty} \to C(X,X)$ given by $n \mapsto \pi_n$ is continuous, where $C(X,X)$ has the supremum topology.
\end{lemma}

\begin{proof}
    Let $n \in \N_{0,\infty}$.  To see that $\pi_n: X \to X$ is continuous, let $\omega \in X$ and $N \in \N$.  If $\nu_A(\omega) \in \N_{0,-\infty}$, then for all $\omega' \in X$ sufficiently close to $\omega$, the points $\omega'$ and $\omega$ agree on $[N]$ and $\nu_A(\omega') = \nu_A(\omega)$.  If, on the other hand, $\nu_A(\omega) = \infty$, then for all $\omega' \in X$ sufficiently close to $\omega$, the points $\omega'$ and $\omega$ agree on $[N]$ and $\nu_A(\omega') \geq N$.  In both cases, we see that $\pi_n(\omega')$ and $\pi_n(\omega)$ agree on $[N]$.  By the topology of the space $X$, this shows that $\pi_n$ is continuous.

    To show that the map $\N_{0,\infty} \to C(X,X)$ given by $n \mapsto \pi_n$ is continuous, it suffices by the topology on $\N_{0,\infty}$ to show that $\lim_{n \to \infty} \pi_n = \pi_\infty$, that is, for all $\eps > 0$ and all sufficiently large $n \in \N$,
    \[\sup_{\omega \in X} d_X\big(\pi_n(\omega), \pi_\infty(\omega) \big) < \eps.\]
    Let $\eps > 0$ and $\omega \in X$. If $\nu_A(\omega) = -\infty$, then for all $n \in \N_0$, $\pi_\infty(\omega) = \pi_n(\omega)$. If $\nu_A(\omega) \in \N_0$, then for all $n \in \N_0$ with $n \geq \nu_A(\omega)$, $\pi_\infty(\omega) = \pi_n(\omega)$. If $\nu_A(\omega) = \infty$, then $\pi_n(\omega)$ and $\pi_\infty(\omega)$ match on $[2^n]$, because they both match $1_A$ on $[2^n]$.  Thus, in all cases, for $n \in \N_0$ sufficiently large, $d_X(\pi_n(\omega), \pi_\infty(\omega)) < \eps$, as was to be shown.
\end{proof}

The following lemma gives finer information on the continuity of $\pi_n$, $n \in \N_0$.

\begin{lemma}
\label{lemma_continuity_of_pins}
    Let $n, \ell \in \N_0$ with $\ell \geq 2^n$.  If $\omega, \xi \in X$ agree on $[\ell]$, then $I_n(\omega) = I_n(\xi)$, and $\pi_n(\omega)$ and $\pi_n(\xi)$ agree on $[\ell]$.
\end{lemma}

\begin{proof}
    Suppose that $\omega, \xi \in X$ agree on $[\ell]$.  We consider the following cases.
    \begin{enumerate}
        \item Case 1: $\nu_A(\omega) = -\infty$ or $\nu_A(\xi) = -\infty$.  In this case, since $\omega$ and $\xi$ agree on $[\ell]$, we have that $\nu_A(\omega) = \nu_A(\xi) = -\infty$ and $I_n(\omega) = I_n(\xi) = \emptyset$.  Since $\omega$ and $\xi$ are unchanged by $\pi_{n}$, we have that $\pi_{n} (\omega)$ and $\pi_{n} (\xi)$ agree on $[\ell]$.

        \item Case 2: $0 \leq \nu_A(\omega) < n$.  By the definition of $\nu_A$, we have that $A \cap [2^{\nu_A(\omega)}] \subseteq \supp(\omega)$ but $A \cap [2^{\nu_A(\omega)+1}] \not\subseteq \supp(\omega)$. Since $2^{\nu_A(\omega)} < 2^{n} \leq \ell$, we have that $[2^{\nu_A(\omega)+1}] \subseteq [\ell]$. Since $\omega$ and $\xi$ agree on $[\ell]$, we see that $\nu_A(\omega) = \nu_A(\xi)$, and hence that $I_n(\omega) = I_n(\xi)$.  Since $\pi_{n}$ changes $\omega$ and $\xi$ on the same interval contained in $[\ell]$, we see that $\pi_{n} (\omega)$ and $\pi_{n} (\xi)$ agree on $[\ell]$, as desired.

        \item Case 3: $0 \leq \nu_A(\xi) < n$. We argue just as in Case 2 to arrive at the same conclusion.

        \item Case 4: $\nu_A(\omega) \geq n$ and $\nu_A(\xi) \geq n$. In this case, we see that $I_n(\omega) = I_n(\xi) = [2^{n}]$.  Since $2^{n} \leq \ell$ and $\omega$ matches $\xi$ on $[\ell]$, we see that $\pi_{n} (\omega)$ and $\pi_{n} (\xi)$ agree on $[\ell]$, as desired.
    \end{enumerate}
    This concludes the casework and the proof of the lemma.
\end{proof}

\subsubsection{The shift-punch system}
\label{sec_punch_system}

The \emph{shift-punch system} $(\odometer \times X, \kappa)$ is the topological skew product system defined by the \emph{shift-punch map} $\kappa: \odometer \times X \to \odometer \times X$,
\[\kappa(z,\omega) \defeq \big(\orotation z, \pi_{\nu_2( \orotation z)}(\sigma \omega) \big).\]
Though not apparent from the notation, the punch maps $\pi_n$, the shift-punch map $\kappa$, and the shift-punch system $(\odometer \times X, \kappa)$ all depend highly on the set $A$.

\begin{lemma}
    The shift-punch system $(\odometer \times X, \kappa)$ is a topological dynamical system.
\end{lemma}

\begin{proof}
    The space $Z \times X$ is compact, so we need only to show that the map $\kappa$ is continuous.  Because $\kappa$ is a skew product, it suffices to check that the map $\odometer \to C(X,X)$ defined by $z \mapsto \pi_{\nu_2(\orotation z)} \circ \sigma$ is continuous.  This follows from the fact that the shift $\sigma: X \to X$ is continuous and that the map $z \mapsto \pi_{\nu_2(\orotation z)}$ is a composition of three continuous maps: $\orotation: \odometer \to \odometer$, $\nu_2: \odometer \to \N_{0,\infty}$, and, by \cref{lemma_continuity_of_family_of_punches}, $\pi: \N_{0,\infty} \to C(X,X)$.
\end{proof}

We are interested primarily in the orbit of the point $(0,1_A)$ under the shift-punch map.  In the shift-punch dynamics, it will be important to keep track of where the punches (changes from 1 to 0) occur.  The following notation will assist with that.  Define $\bpunch^{(0)} = 1_{\N_0}$ and $\apunch^{(0)} = 1_A$, and for $n \in \N$, define $\bpunch^{(n)}, \apunch^{(n)} \in X$ to be such that
\begin{align*}
    (\orotation \times \sigma) \kappa^{n-1}(0,1_A) &= \big(n, \bpunch^{(n)}\big),\\
    \kappa^n (0,1_A) &= \big(n, \apunch^{(n)} \big).
\end{align*}
By the definition of $\kappa$, we see that for all $n \in \N_0$, $\apunch^{(n)} = \pi_{\nu_2(n)} \bpunch^{(n)}$ and $\bpunch^{(n+1)} = \sigma \apunch^{(n)}$.
Thus, the points $\bpunch^{(n)}$ and $\apunch^{(n)}$ are the second coordinates of $\kappa^{n}(0, 1_A)$ before and after, respectively, the punch by $\pi_{\nu_2(n)}$.

The \emph{punch window} at $n \in \N_0$ is defined to be
\[W(n) \defeq n + I_{\nu_2(n)}\big( \bpunch^{(n)} \big),\]
where $W(n)$ is understood to be the empty set if $I_{\nu_2(n)}\big( \bpunch^{(n)} \big)$ is the empty set.  Note that $W(0) = \N_0$; that $W(n)$ is an interval beginning at $n$ of length a divisor of $2^{\nu_2(n)}$ (and, hence, is a dyadic interval); and that $\apunch^{(n)}$ and $1_A$ agree on the interval $W(n) - n$.  Informally, we note that $\supp(\bpunch^{(n)}) \subseteq A-n$ (as $\bpunch^{(n)}$ is the result of the first $n-1$ shift-punches and a single shift) just before $\pi_{\nu_2(n)}$ is applied to change $\bpunch^{(n)}$ on the punch interval $I_{\nu_2(n)}\big( \bpunch^{(n)} \big)$ to yield $\apunch^{(n)}$, the second coordinate of $\kappa^n (0,1_A)$.  The interval $W(n)$ is the interval in the original set $A$ at which that punch occurs.

\begin{lemma}
\label{lemma_explicit_description_of_omega_and_kappa}
    For all $n \in \N_0$,
    \begin{enumerate}
        \item \label{item_explicit_description_of_omega_and_kappa_one}
    \begin{align}
        \label{eqn_set_indicated_by_omega_n}
        \supp(\bpunch^{(n)}) = \bigcap_{m=0}^{n-1} \Big( \big( A - (n-m) \big) \cup \big( \N_0 \setminus \big( W(m)-n \big) \big) \Big),
    \end{align}
        where the empty intersection is interpreted as $\N_0$;

        \item \label{item_explicit_description_of_omega_and_kappa_two}
    \begin{align}
        \label{eqn_set_indicated_by_kappa_n}
        \supp(\apunch^{(n)}) = \bigcap_{m=0}^{n} \Big( \big( A - (n-m) \big) \cup \big( \N_0 \setminus \big( W(m)-n \big) \big) \Big).
    \end{align}
    \end{enumerate}
\end{lemma}

\begin{proof}
    We will prove both \eqref{item_explicit_description_of_omega_and_kappa_one} and \eqref{item_explicit_description_of_omega_and_kappa_two} simultaneously by induction on $n \in \N_0$. The base case $n=0$ follows by set algebra, recalling that, by definition, $\bpunch^{(0)} = 1_{\N_0}$ and $\apunch^{(0)} = 1_A$.
    
    Suppose both \eqref{item_explicit_description_of_omega_and_kappa_one} and \eqref{item_explicit_description_of_omega_and_kappa_two} hold for some $n \in \N_0$. Let $i \in \N_0$.  Since $\bpunch^{(n+1)} = \sigma \apunch^{(n)}$, we see that $\big(\bpunch^{(n+1)} \big)(i) = 1$ if and only if $\big(\apunch^{(n)} \big)(i+1) = 1$. By the inductive hypothesis, we have that $\big(\apunch^{(n)} \big)(i+1) = 1$ if and only if
    \[i+1 \in \bigcap_{m=0}^{n} \Big( \big( A - (n-m) \big) \cup \big( \N_0 \setminus \big( W(m)-n \big) \big) \Big).\]
    The previous line rearranges to
    \begin{align}
    \label{eqn_big_set_description_for_inductive_step}
        i \in \bigcap_{m=0}^{n} \Big( \big( A - (n+1-m) \big) \cup \big( \N_0 \setminus \big( W(m)-(n+1) \big) \big) \Big),
    \end{align}
    which shows \eqref{item_explicit_description_of_omega_and_kappa_one} for $n+1$.

    To see \eqref{item_explicit_description_of_omega_and_kappa_two} for $n+1$, let $i \in \N_0$. Note that by the definition of the map $\pi_{\nu_2(n+1)}$, $\big(\apunch^{(n+1)} \big)(i) = \big(\pi_{\nu_2(n+1)} \bpunch^{(n+1)} \big)(i) = 1$ if and only if $\big(\bpunch^{(n+1)}\big)(i) = 1$ and either
    \begin{itemize}
        \item $i \not\in I_{\nu_2(n+1)}(\bpunch^{(n+1)}) = W(n+1) - (n+1)$, so that $\pi_{\nu_2(n+1)}$ does not modify $\bpunch^{(n+1)}$ at $i$; or
        \item $i \in A$, so that, if $\pi_{\nu_2(n+1)}$ modifies $\bpunch^{(n+1)}$ at $i$, it does not change 1 to 0.
    \end{itemize}
    By the previous paragraph, we have that $\big(\bpunch^{(n+1)}\big)(i) = 1$ if and only if \eqref{eqn_big_set_description_for_inductive_step} holds, and so we see by the previous sentence that $\big(\apunch^{(n+1)} \big)(i) = 1$ if and only if $i$ belongs to the set
    \[\bigcap_{m=0}^{n} \Big( \big( A - (n+1-m) \big) \cup \big( \N_0 \setminus \big( W(m)-(n+1) \big) \big) \Big) \cap \big(A \cup \big(\N_0 \setminus (W(n+1)-(n+1))\big) \big).\]
    The set on the previous line simplifies to the one in \eqref{eqn_set_indicated_by_kappa_n} for $n+1$, demonstrating \eqref{item_explicit_description_of_omega_and_kappa_two} for $n+1$, as desired.
\end{proof}

\begin{lemma}
\label{lemma_second_coords_agree}
    Let $m \in \N_0$.  For all $i \in W(m) - m$, the points $\apunch^{(i)}$ and $\apunch^{(m+i)}$ agree on the interval $W(m) - m - i$.
\end{lemma}

\begin{proof}
The conclusion is trivial for $m = 0$, so suppose $m \in \N$. Let $\ell = |W(m)|$. We must show that for all $i \in [\ell]$, the points $\apunch^{(i)}$ and and $\apunch^{(m+i)}$ agree on $W(m) - m - i = [\ell - i]$.  We will prove this by induction on $i$.  The base case $i = 0$ follows from the definition of $W(m)$: the points $\apunch^{(0)} = 1_A$ and $\apunch^{(m)}$ agree on $W(m) - m = [\ell]$.

Let $i \in [\ell - 1]$, and suppose that $\apunch^{(i)}$ and $\apunch^{(m+i)}$ agree on $[\ell - i]$.  We will show that $\apunch^{(i+1)}$ and $\apunch^{(m+i+1)}$ agree on $[\ell - i - 1]$.  Note that $\apunch^{(i+1)} = \pi_{\nu_2(i+1)} \sigma \apunch^{(i)}$ and $\apunch^{(m+i+1)} = \pi_{\nu_2(m+i+1)} \sigma \apunch^{(m+i)}$.

Since $\apunch^{(i)}$ and $\apunch^{(m+i)}$ agree on $[\ell - i]$, we have that $\sigma \apunch^{(i)}$ and $\sigma \apunch^{(m+i)}$ agree on $[\ell - i - 1]$.  If we show that $\nu_2(i+1) = \nu_2(m+i+1)$ and that $2^{\nu_2(m+i+1)} \leq \ell - i - 1$, then it will follow from \cref{lemma_continuity_of_pins} that $\pi_{\nu_2(i+1)} \sigma \apunch^{(i)}$ and $\pi_{\nu_2(m+i+1)} \sigma \apunch^{(m+i)}$ agree on $[\ell - i - 1]$, concluding the proof of the inductive step.

Since $W(m)$ is a dyadic interval, $m+i+1 \in (W(m) \setminus \{m\}) \cap 2^{\nu_2(m+i+1)} \N$, and $|W(m)| \leq 2^{\nu_2(m)}$, we get from \cref{lemma_basic_dyadic_cubes_facts} that $\nu_2(m+i+1) < \nu_2(m)$.  By properties of the 2-adic valuation, it follows that $\nu_2(m+i+1) = \nu_2(i+1)$. To see that $2^{\nu_2(m+i+1)} \leq \ell - i - 1$, note that $m+i+1+[2^{\nu_2(m+i+1)}]$ is a dyadic interval, intersecting, hence contained in, $W(m)$.  Therefore, $m+i+1+2^{\nu_2(m+i+1)} \leq m + |W(m)| = m + \ell$, which rearranges to the desired inequality, finishing the proof of the inductive step.
\end{proof}

The following lemma shows the ``fractal'' structure of containment amongst the punch windows.

\begin{lemma}
\label{lemma_fractal_windows}
    For all $n, m \in \N_0$, if $m < n$ and $n \in W(m)$, then $W(n)-m = W(n-m)$.
\end{lemma}

\begin{proof}
Let $n, m \in \N_0$ with $m < n$ and $n \in W(m)$.  If $m = 0$, then conclusion is immediate, so suppose $m \geq 1$.  Define $\ell = |W(m)|$.

We claim that $\nu_2(n) = \nu_2(n-m)$ and that $2^{\nu_2(n)} \leq \ell - (n-m)$.  Indeed, since $W(m)$ is a dyadic interval, $n \in (W(m) \setminus\{m\}) \cap 2^{\nu_2(n)} \N$, and $2^{\nu_2(m)} \geq |W(m)|$, we see from \cref{lemma_basic_dyadic_cubes_facts} that $\nu_2(n) < \nu_2(m)$.  By properties of the 2-adic valuation, this implies that $\nu_2(n) = \nu_2(n-m)$.  Since $n + [2^{\nu_2(n)}]$ is a dyadic interval intersecting, and hence contained in, $W(m)$, we see that $n + 2^{\nu_2(n)} \leq m + |W(m)| = m + \ell$.  Rearranging, we have that $2^{\nu_2(n)} \leq \ell - (n-m)$.

Next, we claim that $\bpunch^{(n)}$ and $\bpunch^{(n-m)}$ agree on $[\ell - n+m]$.  Indeed, it follows from \cref{lemma_second_coords_agree} with $i = n-m-1 \in [\ell]$ that the points $\apunch^{(n-m-1)}$ and $\apunch^{(n-1)}$ agree on $[\ell - (n-m-1)]$.  Applying the map $\sigma$, we see that $\bpunch^{(n)} = \sigma \apunch^{(n-1)}$ and $\bpunch^{(n-m)} = \sigma \apunch^{(n-m-1)}$ agree on $[\ell - n+m]$.

To reach the conclusion of the lemma, since $W(n)-m$ and $W(n-m)$ are both (possibly empty) intervals starting at $n-m$, it suffices to show that $|W(n)| = |W(n-m)|$. Recall that $|W(n)|$ is the length of the interval on which $\pi_{\nu_2(n)}$ changes $\bpunch^{(n)}$, and similarly for $|W(n-m)|$ and $\bpunch^{(n-m)}$.  Since $\bpunch^{(n)}$ and $\bpunch^{(n-m)}$ agree on $[\ell - n+m]$, and since $\nu_2(n) = \nu_2(n-m)$ and $2^{\nu_2(n)} \leq \ell - (n-m)$, it follows from \cref{lemma_continuity_of_pins} that $|W(n)| = |W(n-m)|$, as desired.
\end{proof}

\begin{lemma}
\label{lemma_punches_at_new_elements}
    Let $n, \ell \in \N_0$.  If
    \begin{enumerate}
        \item \label{item_n_is_new_position}
        the greatest integer $m \in [n]$ such that $n \in W(m)$ is $0$, and

        \item \label{item_original_A_is_punchable_at_n}
        $n + \big(A \cap [2^{\ell}]\big) \subseteq A$,
    \end{enumerate}
    then $|W(n)| \geq 2^{\min(\ell,\nu_2(n))}$.
\end{lemma}

\begin{proof}
    If $n = 0$, then $W(0) = \N$ and the conclusion holds.  Suppose $n \geq 1$.  It follows from \cref{lemma_explicit_description_of_omega_and_kappa} that $\bpunch^{(n)}$ is equal to $1_{A-n}$ on the set $\N_0 \setminus \bigcup_{m=1}^{n-1} (W(m) - n)$. Indeed, if $i \in \N_0 \setminus \bigcup_{m=1}^{n-1} (W(m) - n) = \bigcap_{m=1}^{n-1} \big( \N_0 \setminus (W(m) - n) \big)$, then $i$ belongs to the set in \eqref{eqn_set_indicated_by_omega_n} if and only if $i \in A-n$. If \eqref{item_n_is_new_position} holds, then the set $\bigcup_{m=1}^{n-1} (W(m) - n)$ contains no non-negative integers.  Therefore, the points $\bpunch^{(n)}$ and $1_{A-n}$ are equal on the set $\N_0$.

    Suppose \eqref{item_n_is_new_position} and \eqref{item_original_A_is_punchable_at_n} hold. Since $\bpunch^{(n)} = 1_{A-n}$ and $n + \big(A \cap [2^{\ell}]\big) \subseteq A$, we see that $\nu_A(\bpunch^{(n)}) \geq \ell$.  It follows from the definition of $W(n)$ that
    \[W(n) = n + [2^{\min(\nu_A(\bpunch^{(n)}),\nu_2(n))}] \supseteq n + [2^{\min(\ell,\nu_2(n))}],\]
    whereby $|W(n)| \geq 2^{\min(\ell,\nu_2(n))}$, as desired.
\end{proof}

For $\ell \in \N_0$, we define
\[L(\ell) \defeq \big\{n \in \N \ \big| \ |W(n)| \geq 2^\ell \big\}\]
to be the set of those positive integers $n$ at which the punch window $W(n)$ is of length at least $2^\ell$.  The following lemma shows that the collection of return time sets of the point $(0,1_A)$ to neighborhoods of itself under the shift-punch map is essentially the same as the collection of sets $L(\ell)$, $\ell \in \N_0$.

\begin{lemma} \leavevmode
\label{lemma_punch_windows_and_return_times}
    \begin{enumerate}
        \item \label{item_Lell_contained_in_return_times}
        For all $\eps > 0$, there exists $\ell \in \N_0$ such that
        \[L(\ell) \subseteq R_{\kappa}\big((0,1_A),B_{\eps}((0,1_A))\big).\]

        \item \label{item_return_times_contained_in_Lell}
        For all $\ell \in \N_0$, there exists $\eps > 0$ such that
        \[R_{\kappa}\big((0,1_A),B_{\eps}((0,1_A))\big) \subseteq L(\ell).\]
    \end{enumerate}
\end{lemma}

\begin{proof}
    \eqref{item_Lell_contained_in_return_times} \
    Let $\eps > 0$.  Choose $\ell \in \N_0$ sufficiently large so that for all $n \in \N$ and all $\omega \in X$, if $\nu_2(n) \geq \ell$ and $\omega$ and $1_A$ agree on $[2^\ell]$, then $d_{\odometer \times X}\big( (n,\omega), (0, 1_A) \big) < \eps$.

    Let $n \in L(\ell)$.  We see from the definition of $W(n)$ that $\nu_2(n) \geq |W(n)| \geq \ell$ and that the points $\apunch^{(n)}$ and $1_A$ agree on $W(n) - n \supseteq [2^\ell]$.  By the choice of $\ell$, we have that $d\big( \kappa^n (0,1_A), (0, 1_A)\big) < \eps$, whereby $n \in R_{\kappa}\big((0,1_A),B_{\eps}((0,1_A))\big)$, as was to be shown.

    \eqref{item_return_times_contained_in_Lell} \
    Let $\ell \in \N_0$.  Choose $\eps > 0$ so that for all pairs of points $(n,\omega) \in \odometer \times X$, if $(n,\omega)$ and $(0,1_A)$ are at a distance of at most $\eps$, then $\nu_2(n) \geq \ell$ and $\omega$ and $1_A$ agree on $[2^{\ell}]$.

    Let $n \in R_{\kappa}\big((0,1_A),B_{\eps}((0,1_A))\big)$.  Since $\kappa^n(0,1_A) = (n, \apunch^{(n)})$ and $(0,1_A)$ are at a distance of at most $\eps$, we have that $\nu_2(n) \geq \ell$ and $\apunch^{(n)}$ and $1_A$ agree on $[2^{\ell}]$. By the definition of $\pi_{\nu_2(n)}$, we have that $\apunch^{(n)} \leq \bpunch^{(n)}$ pointwise.  Since $\apunch^{(n)}$ and $1_A$ agree on $[2^{\ell}]$, we have that $A \cap [2^\ell] \subseteq \supp(\bpunch^{(n)})$.  Thus, $\nu_A(\bpunch^{(n)}) \geq \ell$. It follows that $|W(n)| = 2^{\min(\nu_A(\bpunch^{(n)}),\nu_2(n))} \geq 2^{\ell}$, whereby $n \in L(\ell)$, as was to be shown.
\end{proof}

We are ultimately interested in the result of the punches made on the set $A$.  Thus, we define
\begin{align}
\label{eqn_result_of_set_after_punches}
    B \defeq \bigcap_{n=0}^\infty \Big( \big( A + n \big) \cup \big(\N_0 \setminus W(n) \big) \Big).
\end{align}
Informally, the set $B$ is the set $A$ after all punches have occurred; this is made precise by the equivalent dynamical characterization of the set $B$ given in \cref{lemma_dynamical_description_of_B}. Since $W(0) = \N_0$, we see from the definition that $B \subseteq A$.  Moreover, we see that if $0 \in A$, then $0 \in B$.

\begin{lemma}
\label{lemma_dynamical_description_of_B}
    For all $n \in \N_0$, we have that $\big(\apunch^{(n)}\big)(0) = 1_B(n)$.
\end{lemma}

\begin{proof}
    Let $n \in \N_0$.  It follows by \cref{lemma_explicit_description_of_omega_and_kappa} that $\big(\apunch^{(n)}\big)(0) = 1$ if and only if
    \[0 \in \bigcap_{m=0}^{n} \Big( \big( A - (n-m) \big) \cup \big( \N_0 \setminus \big( W(m)-n \big) \big) \Big).\]
    By set algebra, the previous line holds if and only if
    \[n \in \bigcap_{m=0}^{n} \Big( \big( A +m \big) \cup \big( \N_0 \setminus W(m) \big) \Big).\]
    Since for all $m \in \N_0$ greater than $n$, $W(m) \subseteq \{n+1, n+2, \ldots \}$, we see that the previous line holds if and only if
    \[n \in \bigcap_{m=0}^{\infty} \Big( \big( A +m \big) \cup \big( \N_0 \setminus W(m) \big) \Big) = B,\]
    as was to be shown.
\end{proof}

The following lemma connects the times of return of the point $1_B$ to neighborhoods of itself under the usual shift to the family of sets $L(\ell)$ defined above.

\begin{lemma}
\label{lemma_punch_returns_vs_shift_returns}
    For all $\eps > 0$, there exists $\ell \in \N_0$ such that
    \[L(\ell) \subseteq R_\sigma\big(1_B, B_\eps\big(1_B \big) \big).\]
\end{lemma}

\begin{proof}
Let $\eps > 0$.  Choose $\ell \in \N_0$ so that if two points in $X$ agree on $[2^\ell]$, then they are at a distance of no more than $\eps$.  

Let $n \in L(\ell)$.  By \cref{lemma_second_coords_agree}, for all $i \in [2^\ell] \subseteq W(n) - n$, the points $\apunch^{(i)}$ and $\apunch^{(n+i)}$ agree on $W(n) - n - i$.  In particular, they agree at 0.  It follows from \cref{lemma_dynamical_description_of_B} that for all $i \in [2^\ell]$, $1_B(i) = 1_B(i+n) = 1_{B-n}(i) = (\sigma^{n} 1_B)(i)$.  Since $1_B$ and $\sigma^{n} 1_B$ agree on $[2^\ell]$, we see that $n \in R_\sigma(1_B,B_\eps(1_B))$, as was to be shown.
\end{proof}

\begin{remark}
    If $0 \not\in A$, then $B = \emptyset$.  Indeed, since $0 \not\in A$, for all $n \in \N_0$, the map $\pi_{\nu_2(n)}$ will punch $\bpunch^{(n)}$ at least at the zero coordinate (that is, $n \in W(n)$), changing that coordinate to 0 if it is 1.  Thus, $\apunch^{(n)}(0) = 0$, and according to \cref{lemma_dynamical_description_of_B}, the set $B$ is empty. In this case, we see that $R_\sigma\big(1_B, B_\eps\big(1_B \big) \big) = \N$.  This shows that a result analogous to \cref{lemma_punch_windows_and_return_times} for the point $1_B$ under the shift $\sigma$ does not necessarily hold: it may be that there exists $\ell \in \N_0$ such that for all $\eps > 0$, the set $R_\sigma\big(1_B, B_\eps(1_B ) \big)$ is not contained in $L(\ell)$.
\end{remark}

\subsubsection{Finishing the argument}
\label{sec_proof_that_member_of_synd_idem_filter_is_dcsyndetic}

To finish the proof of \cref{thm_dcs_iff_cs}, we must show that central syndetic sets are dynamically central syndetic.  It suffices by \cref{thm_conditions_for_central_syndetic} and \cref{lemma_upgraded_cssd_property} to show: if $A \subseteq \N$ satisfies the condition in \eqref{eqn_stronger_than_cssd}, then $A$ is dynamically central syndetic. We carry forward all of the notation from the previous sections.

Let $A \subseteq \N$ satisfy the condition in \eqref{eqn_stronger_than_cssd}.  We append $0$ to $A$ and, by a slight abuse of notation, consider $A$ as a subset of $\N_0$.  The key step is to show that the condition in \eqref{eqn_stronger_than_cssd} implies that the sets $L(\ell)$ defined in the previous section are syndetic.

\begin{theorem}
\label{prop_long_punch_times_are_syndetic}
    For all $\ell \in \N_0$, the set $L(\ell)$ is syndetic.
\end{theorem}

\begin{proof}
    Let $\ell \in \N_0$. We will show first that the set $L(\ell)$ is nonempty.  It follows from \eqref{eqn_stronger_than_cssd} that there exists $n \in 2^\ell \N$ such that $n + \big(A \cap [2^\ell]\big) \subseteq A$.  Let $m \in [n]$ be the greatest integer such that $n \in W(m)$.  If $m=0$, then by \cref{lemma_punches_at_new_elements}, we have that $|W(n)| \geq 2^\ell$, whereby $n \in L(\ell)$.  Otherwise, we have that $1 \leq m < n$ and $n \in W(m)$, whereby $\big( W(m) \setminus \{m\} \big) \cap 2^\ell \N \neq \emptyset$.  It follows from \cref{lemma_basic_dyadic_cubes_facts} that $|W(m)| \geq 2^{\ell + 1}$, so $m \in L(\ell)$.

    By \eqref{eqn_stronger_than_cssd}, \cref{lemma_syndetic_characterization}, and the previous paragraph, there exists $k \in \N$ so that
    \begin{align}
    \label{eqn_cube_syndeticity_from_upgraded_membership_property}
        2^\ell \N_0 \cap \bigcap_{a \in A \cap [2^{\ell}]} (A-a) \in \cubes_{2^{k}}^*
    \end{align}
    and so that $L(\ell) \cap [2^k] \neq \emptyset$.  We will show that $L(\ell)$ is syndetic by showing that $L(\ell) \in \cubes_{2^{k+1}}^*$ and appealing again to \cref{lemma_syndetic_characterization}.

    To show that $L(\ell) \in \cubes_{2^{k+1}}^*$, we will show that for all $a \in 2^{k+1}\N_0$, the set $L(\ell) \cap \big(a + ([2^{k+1}]\setminus \{0\})\big)$ is nonempty.  We proceed by induction on $a$.  We have already shown the base case $a=0$. (Recall that, by definition, $L(\ell) \subseteq \N$, so that $L(\ell) \cap [2^k] \neq \emptyset$ implies $L(\ell) \cap ([2^k] \setminus \{0\}) \neq \emptyset$.)

    Let $a \in 2^{k+1}\N$, and define $Q = a + [2^{k+1}]$.  From \eqref{eqn_cube_syndeticity_from_upgraded_membership_property}, there exists
    \[n \in \big( Q \setminus \{a\} \big) \cap 2^\ell \N_0 \cap \bigcap_{a \in A \cap [2^{\ell}]} (A-a).\]
    Let $m \in [n]$ be the greatest integer such that $n \in W(m)$.  If $m=0$, then by \cref{lemma_punches_at_new_elements}, we have that $|W(n)| \geq 2^\ell$, whereby $n \in L(\ell) \cap \big( Q \setminus \{a\} \big)$, as desired. Otherwise, we have that $1 \leq m < n$.

    If $m > a$, then $m \in Q \setminus \{a\}$.  Since $m < n$ and $n \in W(m)$, we have that $\big( W(m) \setminus \{m\} \big) \cap 2^\ell \N \neq \emptyset$.  It follows from \cref{lemma_basic_dyadic_cubes_facts} that $|W(m)| \geq 2^{\ell + 1}$, so $m \in L(\ell) \cap \big( Q \setminus \{a\} \big)$, as desired.

    If, on the other hand, $m \leq a$, then the fact that $W(m)$ and $Q$ are intersecting dyadic intervals (both contain $n$) implies that $m \in 2^{k+1} \N$, that $Q \subseteq W(m)$, and that $Q - m \in \cubes_{2^{k+1}}$. It follows by the induction hypothesis that there exists $q \in Q \setminus \{a\}$ such that $q-m \in L(\ell) \cap (Q-m)$.  Thus, $|W(q-m)| \geq 2^\ell$. Since $q \in W(m)$ and $m \leq a < q$, it follows from \cref{lemma_fractal_windows} that $W(q-m) = W(q) - m$.  Therefore, $|W(q)| \geq 2^\ell$, whereby $q \in L(\ell) \cap (Q \setminus \{a\})$, as desired.
\end{proof}

Let $B \subseteq A$ be the subset of $A$ defined in \eqref{eqn_result_of_set_after_punches} in the previous section.  Since $0 \in A$, we see that $0 \in B$. It follows by combining \cref{lemma_punch_returns_vs_shift_returns} and \cref{prop_long_punch_times_are_syndetic} that the point $1_B$ is uniformly recurrent in the shift $(X,\sigma)$.  Since $0 \in B$, it follows from \cref{thm:simple_equivalent_dcS} that the set $B \setminus \{0\} \subseteq \N$ is dynamically central syndetic.  Therefore, we have that the set $A \setminus \{0\}$ is dynamically central syndetic, concluding the proof of \cref{thm_dcs_iff_cs}.

While we do not have need for it, note that it follows by combining \cref{lemma_punch_windows_and_return_times} and \cref{prop_long_punch_times_are_syndetic} that the point $(0,1_A)$ is uniformly recurrent in the shift-punch system $(\odometer \times X, \kappa)$.

\section{Proofs of Theorems A, B, C, and D}
\label{sec_dps_sets}

In this section, we gather everything to prove the main theorems from the introduction.

\subsection{Combinatorial characterizations: proof of Theorems \ref{mainthm_ds_equivalents} and \ref{intro_mainthm_full_chars_of_dcsyndetic}}
\label{sec_char_of_ds_and_dcs}

It is easiest to prove \cref{intro_mainthm_full_chars_of_dcsyndetic} first.

\begin{proof}[Proof of \cref{intro_mainthm_full_chars_of_dcsyndetic}]
    This follows immediately by combining \cref{thm_conditions_for_central_syndetic,thm_dcs_iff_cs}.
\end{proof}

\begin{proof}[Proof of \cref{mainthm_ds_equivalents}]
    \eqref{item_intro_ds_def_condition} $\iff$ \eqref{item_intro_translate_belongs_to_sif} \ If follows by \cref{lemma_translates_of_dsyndetic_sets} that $A$ is dynamically syndetic if and only if there exists $n \in \N_0$ such that $A-n$ is dynamically central syndetic.  By \cref{intro_mainthm_full_chars_of_dcsyndetic}, a set is dynamically central syndetic if and only if it belongs to a syndetic, idempotent filter.  The desired equivalence is shown by combining these two results.

    \eqref{item_intro_translate_belongs_to_sif} $\implies$ \eqref{item_intro_ds_combo_condition} \ Suppose that there exists $n \in \N_0$ such that $A' \defeq A - n$ belongs to a syndetic, idempotent filter.  We have by \cref{intro_mainthm_full_chars_of_dcsyndetic} that there exists $B' \subseteq A'$ such that for all finite $F' \subseteq B'$, the set $B' \cap \bigcap_{f' \in F'} (B' - f')$ is syndetic.  Put $B = B' + n$, and note that $B$ is a nonempty subset of $A$.  Let $F \subseteq B$ be finite, and note that $\min F \geq n+1$.  We have that $F-n \subseteq B'$ is finite, and hence that the set
    \[\bigcap_{f' \in F-n} (B' - f')\]
    is syndetic.  But this set is contained in $\bigcap_{f \in F} (B - f)$, demonstrating syndeticity, as desired.

    \eqref{item_intro_ds_combo_condition} $\implies$ \eqref{item_intro_translate_belongs_to_sif} \ Let $b \in B$.  We will show that the set $B - b \subseteq A-b$ satisfies: for all finite $F \subseteq B-b$, the set
    \[(B-b) \cap \bigcap_{f \in F} (B - b - f)\]
    is syndetic.  It will follow by \cref{thm_conditions_for_central_syndetic} that the set $A-b$ is contained in a syndetic, idempotent filter.

    Let $F \subseteq B-b$ be finite.  By assumption, since $\{b\} \cup (F + b) \subseteq B$ is finite, we have that the set 
    \[\bigcap_{f \in \{b\} \cup (F + b)} (B - f) = (B-b) \cap \bigcap_{f \in F} (B - b - f)\]
    is syndetic, as was to be shown.
\end{proof}

As a short application of \cref{intro_mainthm_full_chars_of_dcsyndetic}, we will show that thickly syndetic sets are dynamically central syndetic.  This result has appeared in related forms several times in the literature (see \cite[Thm. 1]{Glasner-Weiss-Interpolation}, \cite[Thm. 2.4]{huang_ye_2005}, \cite[Prop. 4.4]{dong_shao_ye_2012}).

\begin{theorem}
\label{ex_ps_star_is_dcs}
    Every thickly syndetic set is dynamically central syndetic. Equivalently, sets of pointwise recurrence are piecewise syndetic.
\end{theorem}

\begin{proof}
By \cref{lemma_condition_on_subfamily_of_ps_star}, the family $\PS^*$ is a syndetic, translation-invariant (hence, idempotent) filter.  It follows from \cref{intro_mainthm_full_chars_of_dcsyndetic} that every member of $\PS^*$ is dynamically central syndetic.

It follows from the previous paragraph that $\PS^* \subseteq \dcsyndetic$.  Taking the dual classes, we see that the family of sets of pointwise recurrence is contained in the family of piecewise syndetic sets, as desired.
\end{proof}

\subsection{Polynomial set recurrence and Brauer configurations: proof of \texorpdfstring{\cref{thm:dcps_multiple_recurrence}}{Theorem H}}
\label{sec_dcps_and_set_recurrence}

In this section, we will show that -- in a very general sense made precise by \cref{thm:dcps_multiple_recurrence} -- sets of pointwise recurrence are sets of recurrence in ergodic theory.  Denote by $\dcthick$ the family of sets of pointwise recurrence.  In what follows, a \emph{commuting probability measure preserving system} $(X, \mu, T_1, \ldots, T_k)$ is a tuple consisting of a probability measure space $(X, \mu)$ together with commuting measure preserving transformations $T_i: X \to X$, $i = 1, \ldots, k$.  The system is \emph{invertible} if the maps $T_i$, $i = 1, \ldots, k$, are invertible.

Denote by $\filtertwo$ the upward closure of the family of all subsets of $\N$ of the form
\begin{align}
    \label{eqn_poly_return_set}
    \polyret \big(\mu, T_i, p_{i,j}, E \big)_{\substack{i = 1, \ldots, k \\ j = 1, \ldots, \ell}} \defeq \left\{ n \in \N \ \middle | \ \mu\left( \bigcap_{j=1}^{\ell} \left( \prod_{i=1}^{k} T_i^{p_{i,j}(n)} \right)^{-1} E \right) > 0 \right\},
\end{align}
where $k, \ell \in \N$, the tuple $(X, \mu, T_1, \ldots, T_k)$ is an invertible, commuting probability measure preserving system, the set $E \subseteq X$ satisfies $\mu(E) > 0$, and, for $1 \leq i \leq k$, $1 \leq j \leq \ell$, the polynomial $p_{i, j} \in \Q[n]$ satisfies $p_{i,j}(\Z) \subseteq \Z$ and $p_{i,j}(0) = 0$.  Thus, the family $\filtertwo$ consists of times of returns of any set of positive measure in any probability measure space under polynomial iterates of a finite collection of invertible, commuting measure preserving transformations.  Members of the dual family, $\familytwo^*$, are called \emph{sets of polynomial multiple measurable recurrence for invertible commuting transformations}.

\begin{lemma}
\label{lemma_localization_is_syndetic_idempotent_filter}
    The family $\filtertwo$ is an $\IP^*$, idempotent filter.
\end{lemma}

\begin{proof}
    By a corollary of the IP Polynomial \Szemeredi{} Theorem of Bergelson and McCutcheon \cite[Thm. 7.12]{Bergelson-McCutcheon-IPpolynomialSzemeredi}, every member of $\filtertwo$ is an IP$^*$ set.  Therefore, the family $\filtertwo$ is IP$^*$.

    To show that $\filtertwo$ is a filter, we must show that the set
    \begin{align}
        \label{eqn_set_to_show_is_filter}
        \polyret \big(\mu, T_i, p_{i,j}, E \big)_{\substack{i = 1, \ldots, k \\ j = 1, \ldots, \ell}} \cap \polyret \big(\nu, S_i, q_{i,j}, F \big)_{\substack{i = 1, \ldots, r \\ j = 1, \ldots, t}}
    \end{align}
    belongs to $\filtertwo$, where $(X, \mu, T_1, \ldots, T_k)$, $p_{i, j} \in \Q[n]$, $E \subseteq X$, $(Y, \nu, S_1, \ldots, S_r)$, $q_{i, j} \in \Q[n]$, and $F \subseteq Y$ are as described above.  Denote by $I$ the identity transformation on both $X$ and $Y$, and consider the commuting probability measure preserving system
    \[\big(X \times Y, \mu \otimes \nu, T_1 \otimes I, \ldots, T_k \otimes I, I \otimes S_1, \ldots, I \otimes S_r \big).\]
    The set $E \times F \subseteq X \times Y$ satisfies $(\mu \otimes \nu)(E \times F) = \mu(E) \nu(F) > 0$. Moreover,
    \begin{gather*}
        \big(\mu \otimes \nu \big) \left( \bigcap_{j_1=1}^{\ell} \bigcap_{j_2=1}^{t} \left( \prod_{i=1}^{k} (T_i \otimes I)^{p_{i,j_1}(n)} \ \cdot \ \prod_{i=1}^{r} (I \otimes S_i)^{q_{i,j_2}(n)}\right)^{-1} (E \times F) \right) = \\
        \big( \mu \otimes \nu \big) \left( \bigcap_{j=1}^{\ell} \left( \prod_{i=1}^{k} (T_i \otimes I)^{p_{i,j}(n)} \right)^{-1} (E \times F) \ \cap \ \bigcap_{j=1}^{t} \left( \prod_{i=1}^{r} (I \otimes S_i)^{q_{i,j}(n)} \right)^{-1} (E \times F) \right) = \\
        \mu\left( \bigcap_{j=1}^{\ell} \left( \prod_{i=1}^{k} T_i^{p_{i,j}(n)} \right)^{-1} E \right) \ \cdot \ \nu\left( \bigcap_{j=1}^{t} \left( \prod_{i=1}^{r} S_i^{q_{i,j}(n)} \right)^{-1} F \right).
    \end{gather*}
    It follows that the set in \eqref{eqn_set_to_show_is_filter} is equal to
    \[\polyret \big( \mu \otimes \nu, U_i, m_{i,j}, E \times F \big)_{\substack{i = 1, \ldots, k+r \\ j = 1, \ldots, \ell t}},\]
    where
    \[U_i \defeq \begin{cases} T_i \otimes I & \text{ if $i \leq k$} \\ I \otimes S_{i-k} & \text{ if $i > k$}\end{cases}, \qquad \text{ and } \qquad m_{i,j} \defeq \begin{cases} p_{i, (j \text{ mod } \ell) + 1} & \text{ if $i \leq k$} \\ q_{i-k, \lceil j / \ell \rceil} & \text{ if $i > k$}\end{cases}.\]
    This shows that the set in \eqref{eqn_set_to_show_is_filter} belongs to $\filtertwo$, as desired.

    Next, we will show that $\filtertwo$ is idempotent. It suffices by \cref{rmk_showing_idempotency} to show that if $D$ is a set of the form in \eqref{eqn_poly_return_set}, then $D \subseteq D - \filtertwo$, that is, for all $m \in D$, the set $D-m \in \filtertwo$.
    
    Let $D$ be a set of the form in \eqref{eqn_poly_return_set}, and let $m \in D$.
    Define
    \[
        E' \defeq \bigcap_{j=1}^{\ell} \left( \prod_{i=1}^{k} T_i^{p_{i,j}(m)} \right)^{-1} E.
    \]
    Since $m \in D$, we have that $\mu(E') > 0$. For $i = 1, \ldots, k$ and $j = 1, \ldots, \ell$, define $q_{i,j} \in \Q[n]$ by
    \[
        q_{i,j}(n) \defeq p_{i,j}(n + m) - p_{i,j}(m).
    \]
    Note that $q_{i,j}(\Z) \subseteq \Z$ and $q_{i,j}(0) = 0$.

    By expanding the definition of $E'$, we get
    \begin{align*}
        \bigcap_{j=1}^{\ell} \left( \prod_{i=1}^{k} T_i^{q_{i,j}(n)} \right)^{-1} E' &= \bigcap_{j=1}^{\ell} \left( \prod_{i=1}^{k} T_i^{p_{i,j}(n+m) - p_{i,j}(m)} \right)^{-1} E' \\
        &= \bigcap_{j=1}^{\ell} \left( \prod_{i=1}^{k} T_i^{p_{i,j}(n+m) - p_{i,j}(m)} \right)^{-1} \left( \bigcap_{j'=1}^{\ell} \left( \prod_{i'=1}^{k} T_{i'}^{p_{i',j'}(m)} \right)^{-1} E \right) \\
        &\subseteq \bigcap_{j=1}^{\ell} \left( \prod_{i=1}^{k} T_i^{p_{i,j}(n+m)} \right)^{-1} E.
    \end{align*}
    It follows that
    \[
        \polyret \big(\mu, T_i, q_{i,j}, E' \big)_{\substack{i = 1, \ldots, k \\ j = 1, \ldots, \ell}} \subseteq \left\{n \in \N \ \middle | \ \mu \left( \bigcap_{j=1}^{\ell} \left( \prod_{i=1}^{k} T_i^{p_{i,j}(n+m)} \right)^{-1} E \right) > 0 \right\} = D - m,
    \]
    whereby $D-m \in \filtertwo$, as desired.
\end{proof}

\begin{theorem}
\label{thm_dcps_is_poly_mult_comm_rec}
    Every set of pointwise recurrence is a $\filtertwo^*$ set, that is, a set of polynomial multiple measurable recurrence for invertible commuting transformations.
\end{theorem}

\begin{proof}
    It follows by \cref{lemma_localization_is_syndetic_idempotent_filter} that the family $\familytwo$ is a syndetic, idempotent filter.  Therefore, by \cref{intro_mainthm_full_chars_of_dcsyndetic}, we have that $\familytwo \subseteq \dcsyndetic$.  Taking the dual, we see $\dcthick \subseteq \familytwo^*$, as desired.
\end{proof}

To prove \cref{thm:dcps_multiple_recurrence} from the introduction, we need only to remove the assumption on invertibility in \cref{thm_dcps_is_poly_mult_comm_rec}.

\begin{proof}[Proof of \cref{thm:dcps_multiple_recurrence}]
    Denote by $\filter$ the upward closure of the family of all subsets of $\N$ of the form in \eqref{eqn_poly_return_set}, where $k, \ell \in \N$, the tuple $(X, \mu, T_1, \ldots, T_k)$ is a commuting probability measure preserving system, the set $E \subseteq X$ satisfies $\mu(E) > 0$, and, for $1 \leq i \leq k$, $1 \leq j \leq \ell$, the polynomial $p_{i, j} \in \Q[n]$ satisfies $p_{i,j}(\N_0) \subseteq \N_0$ and $p_{i,j}(0) = 0$. To prove \cref{thm:dcps_multiple_recurrence}, we must show that $\dcthick \subseteq \family^*$.
    
    Let $A \in \dcthick$ and let $B \in \family$.  The set $B$ contains a set of the form $\polyret \big(\mu, T_i, p_{i,j}, E \big)$ as described in the previous paragraph.  By the measurable natural extension (cf. \cite[Lemma 7.11]{Bergelson-McCutcheon-IPpolynomialSzemeredi}), there exists an invertible, commuting probability measure preserving system $(Y, \nu, S_1, \ldots, S_k)$ and a measurable set $F \subseteq Y$ such that for all $n_{i, j} \in \N_0$, $1 \leq i \leq k$, $1 \leq j \leq \ell$,
    \begin{align}
        \label{eqn_measurable_natural_extension}
        \nu \left( \bigcap_{j=1}^\ell \left( \prod_{i=1}^{k} S_i^{n_{i,j}} \right)^{-1} F \right) = \mu \left( \bigcap_{j=1}^\ell \left( \prod_{i=1}^{k} T_i^{n_{i,j}} \right)^{-1} E \right).
    \end{align}
    By \cref{thm_dcps_is_poly_mult_comm_rec}, there exists $n \in A$ such that
    \[\nu\left( \bigcap_{j=1}^{\ell} \left( \prod_{i=1}^{k} S_i^{p_{i,j}(n)} \right)^{-1} F \right) > 0.\]
    It follows from \eqref{eqn_measurable_natural_extension} that $n \in \polyret \big(\mu, T_i, p_{i,j}, E \big)$, showing that $A \cap B \neq \emptyset$, as desired.
\end{proof}

\begin{remark}\label{remark:after_bergelson_mccutcheon}
The IP polynomial \Szemeredi{} theorem \cite[Thm 7.12]{Bergelson-McCutcheon-IPpolynomialSzemeredi} shows that $\filtertwo$ is an $\IP^*$ family, while \cref{thm_dcps_is_poly_mult_comm_rec} shows that $\filtertwo$ is an $\dcthick^* = \dcsyndetic$ family. Because there is no containment in either direction between the families $\IP^*$ and $\dcsyndetic$, \cref{thm_dcps_is_poly_mult_comm_rec} is neither an upgrade nor an immediate corollary of the IP polynomial \Szemeredi{} theorem.

An upgrade to the conclusion of the IP polynomial \Szemeredi{} theorem is available in the case that the polynomials in \eqref{eqn_poly_return_set} are required to be linear.  In that case, the family $\familytwo$ can be shown to be IP$_0^*$ (every member has nonempty intersection with every set that contains arbitrarily large finite sums sets, as defined in \eqref{eqn_def_of_fs_set}) by a corollary of a result of Furstenberg and Katznelson \cite[Thm. 10.3]{furstenberg_katznelson_1985}.
It is not known whether every IP$_0^*$ set is $\dcsyndetic$; we pose this as an open problem in \cref{quest_ip_zero_star_dynam_synd}. If true, the linear case of \cref{thm_dcps_is_poly_mult_comm_rec} could be seen as a corollary of the result of Furstenberg and Katznelson.
\end{remark}

We conclude this subsection by showing that sets of pointwise recurrence contain some combinatorial patterns, namely polynomial Brauer-type configurations.

\begin{theorem}
\label{thm:Brauer_and_dcPS}
    Let $A \subseteq \N$ be a set of pointwise recurrence. For all $k \in \N$ and all $p_1, p_2, \ldots, p_k \in \Q[x]$ with $p_i(\Z) \subseteq \Z$ and $p_i(0) = 0$, there exist $x, y \in \N$ such that
    \begin{align}
        \label{eqn_poly_brauer_in_dcps}
        x, y, x + p_1(y), \ldots, x + p_k(y) \in A.
    \end{align}
\end{theorem}

\begin{proof}
    By \cref{ex_ps_star_is_dcs}, the set $A$ is piecewise syndetic and hence has positive upper Banach density, i.e. $d^*(A) > 0$. By a standard application of the Furstenberg Correspondence Principle (\cite[Thm. 1.1]{Furstenberg-ErgodicBehavior}, see also \cite[Thm. 3.2.5]{mccutcheon_1999}), there exists a probability measure preserving system $(X, \mu, T)$ and a set $E \subseteq X$ with $\mu(E) \geq d^*(A) > 0$ such that for all $n_1, \ldots, n_k \in \Z$,
    \begin{align*}
        \mu \big(T^{-n_1} E \cap \cdots \cap T^{-n_k} E \big) \leq d^* \big( (A-n_1) \cap \cdots \cap (A-n_k) \big).
    \end{align*}
    Setting $p_0 \equiv 0$, by \cref{thm:dcps_multiple_recurrence}, there exists $y \in A$ such that
    \[
        0 < \mu \big(E \cap T^{-p_1(y)} E \cap \cdots \cap T^{-p_k(y)} E \big) \leq d^* \big( A \cap (A-p_1(y)) \cap \cdots \cap (A-p_k(y)) \big).
    \]
    Therefore, there exist (many) $x \in A$ for which \eqref{eqn_poly_brauer_in_dcps} holds, as was to be shown.
\end{proof}

Combining the ideas behind \cref{ex_ps_star_is_dcs}, \cref{thm:dcps_multiple_recurrence}, and \cref{thm:Brauer_and_dcPS} with some more advanced set family algebra, we show in our companion paper \cite{glasscock_le_2025} how these results continue to hold when ``set of pointwise recurrence'' is replaced by ``the intersection of a set of pointwise recurrence and a dynamically central syndetic set.''

\subsection{Partitioning dynamically central syndetic sets: proof of \texorpdfstring{\cref{mainthm_partition_of_dcs_sets}}{Theorem D}}

\label{sec_partitioning_dcs_sets}

Considering some first examples of dynamically central syndetic sets, one may be led to believe that the intersection of two $\dcsyndetic$ sets is again a $\dcsyndetic$ set.  This is true in some special cases: a result of Furstenberg \cite[Thm. 9.11]{furstenberg_book_1981} implies, for example, that the intersection of the set of times of return of a distal point to a neighborhood of itself with any other $\dcsyndetic$ set is $\dcsyndetic$.  It is not, however, true in general, as we show in \cref{mainthm_partition_of_dcs_sets}.  This theorem is best contextualized as a strengthening of \cite[Thm. 2.13]{bergelson_hindman_strauss_2012}, which (when combined with \cref{thm:simple_equivalent_dcS}) says that $\N$ can be partitioned into infinitely many disjoint $\dcsyndetic$ sets.

\begin{proof}[Proof of \cref{mainthm_partition_of_dcs_sets}]
    Let $A \in \dcsyndetic$. We will prove that $A$ can be partitioned into two disjoint $\dcsyndetic$ sets. 
    Then, by repeatedly partitioning the second set into disjoint $\dcsyndetic$ sets,
    \[A = A_0 \cup A_1 = A_{0} \cup A_{10} \cup A_{11} = A_{0} \cup A_{10} \cup A_{110} \cup A_{111} = \cdots,\]
    we see that $A$ can be partitioned into infinitely many disjoint $\dcsyndetic$ sets.
    Let $(X, T)$ be a minimal system and $U \subseteq X$ be an open neighborhood of $x$ such that $A \supseteq R(x, U)$. We consider two cases.

    Case 1: the point $x$ is an isolated point of $X$. By \cref{lemma_isolated_points_and_minimal_systems}, the system $(X,T)$ is finite and periodic. Let $(\T \defeq \R / \Z, S)$ be the rotation on the one-dimensional torus by an irrational $\alpha$. Let $V_1 = (0, 1/2), V_2 = (1/2, 1)$, which are two disjoint, open subsets of $\T$ whose closures contain $0$. Then $\{x\} \times V_1$ and $\{x\} \times V_2$ are two disjoint, open subsets of $X \times \T$ whose closures contain $(x, 0)$. Note that the system $(X \times \T, T \times S)$ is minimal, and so by \cref{thm:simple_equivalent_dcS}, the sets $R((x, 0), \{x\} \times V_1)$ and $R((x, 0), \{x\} \times V_2)$ are $\dcsyndetic$. Moreover,
    \[
        A \supseteq R(x, U) \supseteq R(x, \{x\}) \supseteq R((x, 0), \{x\} \times V_1) \cup R((x, 0), \{x\} \times V_2).
    \]

    Case 2: the point $x$ is a limit point of $U$. Arguing as in the proof of \cref{thm:simple_equivalent_dcS}, there exists a sequence of disjoint open balls $B_{\delta_k}(x_k)$ in $U$ such that $x_k \to x$ and $\delta_k \to 0$. Take 
    \[
        V_0 = \bigcup_{k=1}^{\infty} B_{\delta_{2k}}(x_{2k}) \text{ and } 
        V_1 = \bigcup_{k=1}^{\infty} B_{\delta_{2k - 1}}(x_{2k - 1}).
    \]
    Then $V_0, V_1$ are disjoint open subsets of $U$ satisfying $x \in \overline{V_0} \cap \overline{V_1}$. By \cref{thm:simple_equivalent_dcS}, the sets $R(x, V_0)$ and $R(x, V_1)$ are $\dcsyndetic$ sets and, by construction, they are disjoint subsets of $A$. Put $A_0 \defeq R(x,V_0)$ and $A_1 \defeq A \setminus R(x,V_0)$ to see that $A = A_0 \cup A_1$ is a disjoint partition of $A$ into $\dcsyndetic$ sets, as desired.

    Since $\N$ is a dynamically central syndetic set, there exists a partition $\N = A_0 \cup A_1$ into two disjoint dynamically central syndetic sets $A_0, A_1$.  For any set of pointwise recurrence $B \subseteq \N$, we have the partition $B = (B \cap A_0) \cup (B \cap A_1)$.  Since $(B \cap A_0) \cap A_1 = \emptyset$, the set $B \cap A_0$ is not a set of pointwise recurrence, and the same holds for $B \cap A_1$.
\end{proof}

\section{Open questions}
\label{sec:open_questions}

We collect below a number of open questions and directions suggested by the results in this paper.

\begin{question}
\label{quest_generalizations_to_other_grps}
    To which other group and semigroup actions do the main results in this paper generalize?  For example, a generalization of \cref{intro_mainthm_full_chars_of_dcsyndetic} to $\Z$ might read: a subset of $\Z$ is dynamically central syndetic if and only if it is a member of a syndetic, idempotent filter on $\Z$.
\end{question}

Before attempting to answer \cref{quest_generalizations_to_other_grps}, it would be necessary to decide how the family of dynamically central syndetic subsets of $\Z$ is defined.  We discuss this at more length in point \eqref{item_generalization_to_z} at the beginning of \cref{sec_dsyndetic}.\\

By definition, a subset of $\N$ is dynamically central syndetic if it contains a set of the form $R(x,U)$, where $(X,T)$ is a minimal system, $x \in X$, and $U \subseteq X$ is an open neighborhood of $x$.  It can happen, however, that a set of the form $R(x,U)$ where $x \not \in \overline{U}$ is dynamically central syndetic.  Recall example \eqref{item_dcs_two_adic_example} in \cref{sec_dsyndetic}: the set $A$ of positive integers with even 2-adic valuation is dynamically central syndetic because the point $1_{A \cup \{0\}}$ is uniformly recurrent in the full shift $(\{0,1\}^{\N_0},\sigma)$ and $A = R_\sigma(1_{A \cup \{0\}}, [1]_0)$ where $1_{A \cup \{0\}} \in [1]_0$. Since $\{0,1\}^{\N_0}$ is the disjoint union of the cylinders $[0]_0$ and $[1]_0$, by properties of the 2-adic valuation, we see that
\[A / 2 = \N \setminus A = R_\sigma(1_{A \cup \{0\}}, [0]_0).\]
By \cref{lemma_dilates_of_dsyndetic_sets}, this set is dynamically central syndetic despite the fact that $1_{A \cup \{0\}} \not\in [0]_0$.  This leads us naturally to the following question.

\begin{question}
    \label{quest_nec_and_suff_for_rxu_dcs}
    Let $(X,T)$ be a minimal system, $x \in X$, and $U \subseteq X$ be nonempty and open.  Give necessary and sufficient conditions for $x$ and $U$ so that the set $R(x,U)$ is dynamically central syndetic.
\end{question}

There is precedent for an answer to a related question in the family of central sets (cf. \cite[Def. 8.3]{furstenberg_book_1981}): \emph{If $U \subseteq X$ is clopen, then the set $R(x,U)$ is central if and only if $U$ contains a point $y$ proximal to $x$, that is, $\inf_{n \in \N} d_X(T^n x, T^n y) = 0$.}  Indeed, it is a fact \cite[Lem. 19.22, Thms. 19.23, 19.25]{hindman_strauss_book_2012} that a pair of points $x, y \in X$ is proximal if and only if there exists an idempotent ultrafilter $p \in \beta \N$ such that $T^p x = y$.  If $R(x,U)$ is central, then it belongs to a minimal, idempotent ultrafilter $p$, and so $T^p x \in \overline{U} = U$.  On the other hand, if $x$ is proximal to a point $y \in U$, then there is a minimal, idempotent ultrafilter $p$ such that $T^p x = y$, and so $R(x,U) \in p$.  It seems plausible that an analogous answer to \cref{quest_nec_and_suff_for_rxu_dcs} could be given in terms of some kind of proximality.\\
  
Our main result, \cref{intro_mainthm_full_chars_of_dcsyndetic}, gives that if $A \subseteq \N$ belongs to a syndetic, idempotent filter, then there exists a minimal system $(X,T)$, a point $x \in X$, and an open set $U \subseteq X$ containing $x$ such that $R(x,U) \subseteq A$.  The proof of \cref{intro_mainthm_full_chars_of_dcsyndetic}, however, does not seem to give any indication as to the dynamical nature of the system $(X,T)$.  When, for example, can it be guaranteed that $(X,T)$ is totally minimal, weakly mixing, distal, point distal, a nilsystem, or equicontinuous?  Here, for example, is a concrete question along these lines.  A point $x$ is called \emph{distal} if it is proximal only to itself.

\begin{question}
\label{question_cstar_idempo_filter_implies_pt_distal}
    Does every member of an $\IP^*$, idempotent filter contain a set of the form $R(x,U)$, where $(X,T)$ is a minimal system, $x \in X$ is a distal point, and $U \subseteq X$ is an open neighborhood of $x$?
\end{question}

\cref{question_cstar_idempo_filter_implies_pt_distal} is motivated by a result of Furstenberg \cite[Thm 9.11]{furstenberg_book_1981}: a point $x \in X$ is a distal point if and only if for all open neighborhoods $U \subseteq X$ of $x$, the set $R(x,U)$ belongs to $\IP^*$ if and only if for all open neighborhoods $U \subseteq X$ of $x$, the set $R(x,U)$ belongs to $\central^*$ (the family of sets that have nonempty intersection with all central sets).  Note that the family of thickly syndetic sets, $\PS^*$, is a $\central^*$, idempotent filter.  Thus, an ostensibly easier variant of \cref{question_cstar_idempo_filter_implies_pt_distal} is attained by replacing $\IP^*$ with $\PS^*$.

While IP sets are sets of pointwise recurrence for distal points, we show in the companion paper \cite{glasscock_le_2025} that sets of pointwise recurrence need not be IP.  A positive answer to \cref{question_cstar_idempo_filter_implies_pt_distal} would show, via the following lemma, that a set of pointwise recurrence for distal points must be a translate of an IP set.

\begin{lemma}
    Suppose the answer to \cref{question_cstar_idempo_filter_implies_pt_distal} is positive. If $A \subseteq \N$ is a set of pointwise recurrence for all distal points (ie., for all systems $(X,T)$, all distal $x \in X$, and all open $U \subseteq X$ containing $x$, the set $A \cap R(x,U) \neq \emptyset$), then there exists $n \in \N$ such that $A - n \in \IP$.
\end{lemma}

\begin{proof}
    Consider the family
    \begin{align*}
        \class &\defeq \big\{ A \subseteq \N \ \big | \ \exists \ \ell \in \N, \ \ A-\ell \in \IP \big\}.
    \end{align*}
    Using the fact that the family $\IP$ is partition regular, it is not hard to show that
    \begin{align*}
        \class^* &= \big\{ B \subseteq \N \ \big | \ \forall \ \ell \in \N, \ B \cap (B-1) \cap \cdots \cap (B-\ell) \in \IP^* \big\}.
    \end{align*}
    The family $\class^*$ is an $\IP^*$, translation-invariant (hence, idempotent) filter.

    Suppose $A \subseteq \N$ is a set of pointwise recurrence for all distal points.  We will use the positive answer to \cref{question_cstar_idempo_filter_implies_pt_distal} to show that $A \in \class$ and conclude the proof of the lemma.

    To show that $A \in \class$, we will show that $A \in (\class^*)^*$.  Let $B \in \class^*$.  We must show that $A \cap B \neq \emptyset$.  By the positive answer to \cref{question_cstar_idempo_filter_implies_pt_distal}, the set $B$ contains the times of returns of a distal point to a neighborhood of itself.  By assumption, the set $A \cap B \neq \emptyset$, as was to be shown.
\end{proof}

It was shown in \cref{lemma_translates_of_dsyndetic_sets} \eqref{item_dcs_translates_of_dcs_are_dcs} that the family $\dcsyndetic$ is idempotent: $\dcsyndetic \subseteq \dcsyndetic + \dcsyndetic$.  Is this an equality, and, if not, how close is it to one?

\begin{question}
\label{quest_is_dcs_plus_dcs_equal_to_dcs}
    Is it the case that
    \[\dcsyndetic = \dcsyndetic + \dcsyndetic?\]
    Because it is known that $\dcsyndetic \subseteq \dcsyndetic + \dcsyndetic$, it is equivalent to ask: If $A \subseteq \N$ is such that the set
    \[A - \dcsyndetic \defeq \big\{ n \in \N \ \big | \ A - n \in \dcsyndetic \big\}\]
    is dynamically central syndetic, is $A$ dynamically central syndetic?    
\end{question}

We were not able to show that a set $A \subseteq \N$ is dynamically central syndetic even under the assumption that every shift $A - n$, $n \in \N$, is dynamically central syndetic.  It may be the case that a set with this property is actually thickly syndetic.

The dual form of \cref{quest_is_dcs_plus_dcs_equal_to_dcs} is equally appealing.
Recall that $\dcthick$ denotes the family of sets of pointwise recurrence.
By \cref{lemma_filter_algebra} \eqref{item_sum_dual}, the equivalent, dual form asks: \emph{Is $\dcthick \subseteq \dcthick + \dcthick$, that is, is the family of sets of pointwise recurrence idempotent?}\\

It is simple to see from the definitions that if $A \subseteq \N$ is a set of pointwise recurrence, then for all dynamically central syndetic sets $B \subseteq \N$, there exists $n \in A$ such that the set $B-n$ is dynamically central syndetic.  Does the converse hold, that is, does this property characterize being a set of pointwise recurrence?

\begin{question}
\label{quest_alt_char_of_dcthick}
    Is it true that a set $A \subseteq \N$ is a set of pointwise recurrence if and only if for all $B \in \dcsyndetic$, there exists $n \in A$ such that $B-n \in \dcsyndetic$?
\end{question}

If we denote by $\dcsyndetic - \dcsyndetic$ the family of sets of the form $B - \dcsyndetic$ where $B \in \dcsyndetic$, then \cref{quest_alt_char_of_dcthick} asks whether or not $\dcthick = (\dcsyndetic - \dcsyndetic)^*$.  The equivalent, dual form of this equality is $\dcsyndetic = \dcsyndetic - \dcsyndetic.$ Thus, an equivalent form of \cref{quest_alt_char_of_dcthick} asks whether or not a  `4set $A \subseteq \N$ is dynamically central syndetic if and only if there exists $B \in \dcsyndetic$ such that $B - \dcsyndetic \subseteq A$.  An answer to \cref{quest_nec_and_suff_for_rxu_dcs} would likely be very helpful in addressing this question.\\

The final question pertains to the possibility of underlying global dynamical structure in sets which are combinatorially very large. A subset of $\N$ is IP$_0$ if it contains arbitrarily large finite sums sets, configurations of the form in \eqref{eqn_def_of_fs_set}.  A set is IP$_0^*$ if it has nonempty intersection with every IP$_0$ set.  Since work of Furstenberg and Katznelson \cite{furstenberg_katznelson_1985}, the families of IP$_0$ and IP sets have played an important role in combinatorial refinements of the multiple recurrence results behind \Szemeredi{}'s theorem.

\begin{question}
\label{quest_ip_zero_star_dynam_synd}
    Are IP$_0^*$ sets dynamically (central) syndetic?
\end{question}

There is some precedent for a positive answer if we seek ``local'', instead of global, dynamical structure.  It is a result of Bergelson, Furstenberg, and Weiss \cite{bergelson_furstenberg_weiss_2006} that $\Delta_0^*$ sets -- sets that have nonempty intersection with all sets that contain arbitrarily large, finite difference sets -- are piecewise Bohr$_0$, that is, contain the intersection of a Bohr$_0$ set and and a thick set.  Host and Kra \cite{Host-Kra_nilbohr} generalized this result, showing that members of the family SG$_d^*$ -- a related class of combinatorially large sets -- are piecewise nil-Bohr$_0$.

Our work gives a positive answer to the very first case of \cref{quest_ip_zero_star_dynam_synd}.  A set is IP$_2^*$ if for all $x, y \in \N$, at least one of $x$, $y$, or $x+y$ belong to the set.  It follows from \cref{thm:Brauer_and_dcPS} that every set of pointwise recurrence (and, hence, every dynamically thick set) contains a configuration of the form $\{x, y, x+y\}$.  The equivalent, dual statement -- that every IP$_2^*$ set is dynamically central syndetic -- answers the first non-trivial case of \cref{quest_ip_zero_star_dynam_synd}.  The next case -- Are IP$_3^*$ sets dynamically syndetic? -- is open, as is the ostensibly much simpler question: \emph{If $A \subseteq \N$ is an IP$_3^*$ set, does there exists $n \in \N$ such that $A \cap (A-n)$ is syndetic?}  A strengthening of the dual form of this line of questioning is formulated as Question 7.5 in \cite{glasscock_le_2025}.

\bibliographystyle{abbrv}
\bibliography{dynthick}

\bigskip
\bigskip
\footnotesize
\noindent
Daniel Glasscock\\
\textsc{University of Massachusetts Lowell}\par\nopagebreak
\noindent
\href{mailto:daniel_glasscock@uml.edu}
{\texttt{daniel{\_}glasscock@uml.edu}}

\bigskip
\noindent
Anh N. Le\\
\textsc{University of Denver}\par\nopagebreak
\noindent
\href{mailto:daniel_glasscock@uml.edu}
{\texttt{anh.n.le@du.edu}}

\end{document}